\documentclass{amsart}

\usepackage[english]{babel}
\usepackage[utf8]{inputenc}

\usepackage{amsmath,amssymb,amsthm,mathtools}
\usepackage[linktocpage,unicode]{hyperref}

\usepackage{mathabx}

\usepackage{tikz}

\hypersetup{
    pdftoolbar=true,        
    pdfmenubar=true,        
    pdffitwindow=false,     
    pdfstartview={FitH},    
    pdftitle={},
    pdfauthor={},
    pdfsubject={},
    pdfkeywords={},
    pdfnewwindow=true,      
    colorlinks=true,       
    linkcolor=red,          
    citecolor=blue,        
    filecolor=magenta,      
    urlcolor=cyan,           
    unicode=true          
}

\newtheorem{theo}{Theorem}[section]
\newtheorem{coro}[theo]{Corollary}
\newtheorem{lemm}[theo]{Lemma}
\newtheorem{prop}[theo]{Proposition}

\theoremstyle{definition}
\newtheorem{defi}{Definition}
\newtheorem{rema}{Remark}

\DeclareMathOperator{\supp}{supp}
\DeclareMathOperator{\Fix}{Fix}
\DeclareMathOperator{\rank}{rk}
\DeclareMathOperator{\rk}{rk}
\DeclareMathOperator{\Cat}{Cat}
\DeclareMathOperator{\Nar}{Nar}

\DeclareMathOperator{\Inv}{Inv}
\DeclareMathOperator{\sqr}{sqr}
\DeclareMathOperator{\nir}{nir}

\newcommand{\lldot}{\mathrel{ \mathrlap{\ll}{\;\;\cdot} }}
\newcommand{\sqsubsetdot}{\mathrel{ \mathrlap{\sqsubset}{\;\cdot\;} }}

\newcommand{\F}{{\bf F}}
\renewcommand{\H}{{\bf H}}
\newcommand{\M}{{\bf M}}
\newcommand{\I}{{\bf I}}

\begin{document}
\date{}
\author{Philippe Biane}
\address{Laboratoire d'Informatique Gaspard Monge (LIGM), Université Paris-Est Marne-la-Vallée (UPEM), CNRS}
\author{Matthieu Josuat-Vergès}
\title[Noncrossing part., Bruhat order and the cluster complex]{Noncrossing partitions, Bruhat order and the cluster complex}


\begin{abstract}
We introduce two order relations on finite Coxeter groups which refine the absolute and the Bruhat order, and establish some of their main properties. In particular we study the restriction of these orders to noncrossing partitions and show that the intervals for these orders can be enumerated in terms of the cluster complex.
The properties of our orders permit to revisit several results in Coxeter combinatorics, such as the Chapoton triangles and how they are related, the enumeration of reflections with full support, the bijections between clusters and noncrossing partitions.
\end{abstract}

\maketitle

\section{Introduction}
Let $W$ be a Coxeter group with $S$ a simple system of generators. There exists several natural order relations on $W$, namely the left or right weak order, the Bruhat order and the absolute order (this last order is associated to the length function with respect to the generating set of all reflections, see below). Two elements $v,w\in W$ such that $vw^{-1}$ is a reflection are always comparable with respect to both the absolute and  the Bruhat order. In this paper we  introduce two order relations on $W$, which we denote by $\sqsubset$ and $\ll$, which encode this situation, namely for any pair $v,w$ as above, such that $v< w$ (here $<$ denotes the absolute order), we  define $v\sqsubsetdot w$ if $v<_Bw$ (here $<_B$ is the Bruhat order) and $v\lldot w$ if $w<_B v$, then extend $\sqsubsetdot$ and $\lldot$ to order relations on $W$ by transitivity. 
We believe that these two orders are important tools for understanding noncrossing partitions, clusters, and their interrelations.  Bessis~\cite[Section~6.4]{bessis} suggested to study how the different orders on $W$ are related. The present paper can be considered as a first step in this direction.

Some versions of the orders $\sqsubset$ and $\ll$ were considered before this work. An  order called $\ll$ on the set of classical noncrossing partitions  was introduced independently by Belinschi and Nica \cite{belinschinica,nica} in the context of noncommutative probabilities, and by Senato and Petrullo \cite{petrullo} in order to study Kerov polynomials. The notion of noncrossing partition can be defined in terms of the geometry of the symmetric group \cite{biane} and it has been extended to general Coxeter systems: a set of noncrossing partitions can be associated to some  Coxeter element $c$ in $W$ as $ NC(W,c)=\{w\in W\,|\,w\leq c\}$ (see \cite{bessis,bradywatt}). The classical case corresponds to $c$ being the cycle $(1,\ldots,n)$ in the symmetric group $S_n$.  In this context, the order $\sqsubset$ on $NC(W,c)$ was introduced by the second author in \cite{josuat}, with a different definition, in order to give a refined enumeration of maximal chains in $ NC(W,c)$.

After defining  the two order relations $\sqsubset$ and $\ll$,  we will consider the restriction of these orders to the set of noncrossing partitions  and give a more direct characterization of the pairs $v,w\in NC(W,c)$ with $v\sqsubset w$ or $v\ll w$. We will also introduce interval partitions for arbitrary finite Coxeter groups, which generalize the classical interval partitions. These partitions play an important role in this study.
Then we will consider intervals for the two orders. These turn out to be closely related to the cluster complex of  Fomin and Zelevinsky~\cite{fominzelevinsky}. Originated from the theory of cluster algebras, the cluster complex is a simplicial complex with vertex set the almost positive roots of $W$ (see Section~\ref{cluster_sec} for details), associated to a standard Coxeter element in a finite Coxeter group.  Connections between noncrossing partitions and the cluster complex were first observed via an identity called the $F=M$ theorem,  conjectured by Chapoton~\cite{chapoton1}. We will see that the introduction of the two orders $\sqsubset$ and $\ll$
sheds new light on these relations. In particular an explicit bijection between the facets of the cluster complex and the noncrossing partitions associated with the same Coxeter element was given by Reading \cite{reading}, using the notion of $c$-sortable elements.  Another bijection was given in \cite{bradywatt2} in the case of bipartite Coxeter elements. We will recast this last bijection in terms of  the orders  $\sqsubset$ and $\ll$, which will allow us to extend it to arbitrary standard Coxeter elements, using the definition of the cluster complex by Reading. We will also give a bijection between the intervals for $\ll$ and faces of the positive cluster complex. We also show that intervals of height $k$ for  $\sqsubset$ are equienumerated with the faces of the cluster complex of size $n-k$ and give a bijective proof in the case where $c$ is a bipartite Coxeter element. Finally the orders $\ll$ and $\sqsubset$ will allow us to revisit the Chapoton triangles, to give new proofs of their properties and to refine them.

This paper is organized as follows. In Section~\ref{reflect_sec} we recall basic facts about finite reflections groups, root systems and Coxeter elements. In Section~\ref{noncrois_sec} we define the noncrossing partitions and recall their main properties. We also prove Proposition~\ref{krew1}, showing a relation between  between the Bruhat order and the Kreweras complement on $NC(W,c)$, which plays a crucial role in this work.  Section~\ref{order_sec} is the central part of this paper, in it we introduce the two order relations $\sqsubset$ and $\ll$, which are the main subject of this paper, we characterize these relations and obtain some  of their basic properties. We give a few examples in low rank in Section~\ref{examples_sec}. In Section~\ref{cluster_sec} we study the cluster complex associated to a standard Coxeter element, as defined by Reading \cite{reading}. This simplicial complex is defined with the help of a binary relation for which we give a new characterization.  We also recall facts about nonnesting partitions and Chapoton triangles. 
 In Section~\ref{interval_cluster_sec} we give some enumerative properties of  the intervals of the two order relations which exhibit several  connections with the cluster complex. In the following section, using the orders  $\sqsubset$ and $\ll$, we generalize the bijection of \cite{athanasiadisbradymaccammondwatt} between noncrossing partitions and maximal faces of the cluster complex  to encompass arbitrary standard Coxeter elements, then we use this to give a bijection between intervals for the order $\ll$ and faces of the positive cluster complex. The final section is devoted to some properties of the Chapoton triangles.
  
\section{Finite Coxeter or real reflection groups} \label{reflect_sec}

We fix notations, recall some basic facts about real reflection groups and  refer to \cite{bjornerbrenti,humphreys} for general information about these.

\subsection{Finite reflection groups, roots, reflections and inversions}

\subsubsection{Basic definitions}
Let $V$ be a finite dimensional euclidian space and $W\subset O(V)$ be a finite real reflection group.  
Let $S=\{s_1,\ldots, s_n\}$ be a simple system of reflections for $W$. 
We denote by $\ell$ the length function associated to $S$.  The set of all reflections in $W$ is denoted by $T$. 
A root is a unit vector normal to the hyperplane fixed by some $t\in T$.  The choice of the simple system $S$ imposes a way to split roots into positive and negative roots, and each reflection $t\in T$ has an associated positive root  which we denote $r(t)$. The set of positive roots $\Pi$ is thus in bijection with $T$. The set of negative roots  is $-\Pi$ and the set of all roots  is $\Pi\cup (-\Pi)$.  The fundamental chamber $C$ is the dual cone of the positive span of positive roots.

The {\it Bruhat order}, denoted by $\leq _B$, is defined as the transitive closure of the covering relations $v\lessdot _Bw$ if $vw^{-1}\in T$ and $\ell(v)<\ell(w)$.
A {\it left inversion} (respectively, {\it right inversion}) of $w\in W$ is a $t\in T$ such that $\ell(tw)<\ell(w)$ (respectively, $\ell(wt)<\ell(w)$).
The set of left (respectively, right) inversions is denoted by $\Inv_L(w)$ (respectively, $\Inv_R(w)$). Note that a bijection
$\Inv_L(w) \to \Inv_R(w)$ is given by $t \mapsto w^{-1} t w$.

\begin{prop}[\cite{bjornerbrenti}, Proposition~4.4.6] \label{inv_rootimages}
For $w\in W$ and $t\in T$, we have: 
\begin{align}
   w\big( r(t) \big) = 
    \begin{cases}
      -r( wtw^{-1} ) & \text{ if } t \text{ is a right inversion of } w, \\
       r( wtw^{-1} ) & \text{otherwise}.
    \end{cases}
\end{align}
In particular, $t$ is a right inversion of $w$ if and only if $w( r(t) ) \in (-\Pi)$ and a left inversion if and only if $w^{-1}( r(t) ) \in (-\Pi)$.
\end{prop}

\begin{coro}\label{refcommute}
If $t_1,t_2\in T$ commute then $t_1\in \Inv_R(w)\Leftrightarrow t_1\in \Inv_R(wt_2)$ i.e.
$w\leq_Bwt_1 \Leftrightarrow wt_2\leq_Bwt_2t_1$.
\end{coro}
If a simple reflection $s\in S$ is a right (respectively, left) inversion of $w$, it is called a right (respectively, left) {\it descent} of $w$.

If $w=s_{i_1}s_{i_2}\ldots s_{i_r}$ is a reduced expression for $w$ then the $r$ left inversions of $w$ are the reflections of the form
\begin{equation}\label{leftinv}
(s_{i_1}s_{i_2}\ldots s_{i_{l-1}}) s_{i_l}(s_{i_1}s_{i_2}\ldots s_{i_{l-1}})^{-1}, \qquad l=1,\ldots,r,
\end{equation}
and of course there is a similar formula for right inversions.

The {\it support} of $w\in W$, denoted by $\supp(w)$, is the subset of $S$ containing the simple reflections appearing in some 
reduced expression of $w$. It does not depend on the chosen reduced expression since any two of them are related by
a sequence of braid moves (see \cite{bjornerbrenti}). Equivalently, $\supp(w)$  is the smallest $J$ such that $w\in W_J$, see Section~\ref{Parasubg}.
We say that $w$ has {\it full support}, or that $w$ is {\it full}, if $\supp(w)=S$.
Using the explicit formula in \eqref{leftinv}, we have:
\[
  \supp(t) \subset \supp(w)
\]
if $t$ is a right or left inversion of $w$.

\begin{rema} \label{positivetosimple}
A root 
$r\in\Pi$ is in $\Delta$ iff it cannot be written as a sum $\sum_{r\in \Pi} c_r r$ where the coefficients $c_v$ are $\geq 0$
and at least two of them are nonzero. This characterizes the set of simple roots in $\Pi$.
Also, the simple roots form the unique set of $n$ positive roots having the property that the scalar product of any pair is nonpositive, see~\cite{humphreys}.
\end{rema}

\subsubsection{Coxeter elements}
A {\sl standard Coxeter element} in $(W,S)$  is some product $s_{i_1} \cdots s_{i_n}$ of all the simple reflections in some order. It is known that all standard Coxeter elements are conjugate in $W$, but in general they do not form a full conjugacy class. An element which is conjugated to some standard Coxeter element is called a Coxeter element. In this paper we will mostly consider standard Coxeter elements.

\begin{lemm} \label{scs_moves}
If $c$ is a standard Coxeter element and $s\in S$ a right or left descent of $c$, then $scs$ is also a standard Coxeter element.
All standard Coxeter elements are connected to each other via a sequence of such transformations.
\end{lemm}

See \cite[Section~3.16]{humphreys} for a proof.

Let $S=S_+\cup S_-$ be a partition such that all $s_i$ in $S_+$ commute and all $s_i$ in $S_-$ commute (such a partition always exists), then the standard Coxeter element $c=c_+c_-$ where $c_\pm=\prod_{s\in S_\pm}s$ is called a {\sl bipartite Coxeter element}.

\subsubsection{Absolute length and the absolute order}

  The {\sl absolute length} is the length function associated to the generating set $T$: 
\begin{equation}\label{absolute}
\ell_T(w) =\min\{ k\geq 0\,|\, w\ \text{can be expressed  as  a product of $k$ reflections}\}.
\end{equation}
This quantity has a geometric interpretation (see \cite[Proposition~2.2]{bradywatt}):
\begin{equation} \label{absolutelengthfixdim}
\ell_T(w) = n - \dim(\Fix(w))
\end{equation}
where $\Fix(w)=\ker(w-I)$.

We  call a  factorization $v=v_1\ldots v_k$ in $W$ {\sl minimal} if 
\begin{equation}\label{minfact}
 \ell_T(v)=\ell_T(v_1)+\ldots+\ell_T(v_k).
\end{equation}
The following elementary lemma is well known, cf.~\cite{bessis}.

\begin{lemm}\label{elemlemm}
Let $v=v_1\ldots v_k$ be a minimal factorization, then for any subsequence $i_1\ldots i_l$ with $1\leq i_1<\ldots <i_l\leq k$ the factorization $v_{i_1}\ldots v_{i_l}$ is minimal, moreover $v_{i_1}\ldots v_{i_l}\leq v$.
\end{lemm}

\begin{proof} 
If $i_1=1,i_2=2,\ldots ,i_l=l$ the statement is a simple consequence of the triangle inequality for $\ell_T$. In the general case observe that for any $i<k$ the factorization 
$v=v_1\ldots v_{i-1}\hat v_{i}\hat v_{i+1}v_{i+2}\ldots v_k$ with 
$\hat v_{i}=v_{i+1}$ and $\hat v_{i+1}=v_{i+1}^{-1}v_iv_{i+1}$ is again a minimal. Using this observation we can move successively $v_{i_1},v_{i_2},\ldots,v_{i_l}$ to the beginning and reduce to the first case.
\end{proof}

\begin{rema}The preceding lemma implies that, contrary to the case of reduced decomposition into  simple reflections, in the case of a minimal factorization into reflections  every subword of a reduced word is reduced. In order to avoid confusion, in this paper we will reserve the expressions ``reduced word" and      ``subword" to the case of factorizations into simple reflections, related to the Bruhat order and reserve the expression ``minimal factorization" to the case of factorizations satisfying  (\ref{minfact}).
\end{rema}

The absolute length (\ref{absolute}) allows to define an order relation on $W$, the {\sl absolute order} denoted here by $\leq$:
\[
  v\leq w\quad \text{if}\quad  \ell_T(w)= \ell_T(v^{-1}w)+\ell_T(v).
\]
In particular a cover relation for this order, denoted by $v\lessdot w$, holds if and only if $vw^{-1}\in T$ and $\ell_T(v)=\ell_T(w)-1$. The following properties of the absolute order are immediate or follow directly from Lemma \ref{elemlemm}.

\begin{prop}\label{propelem}
The absolute order is invariant under conjugation and inversion, namely for all $u,v,w\in W$ one has
\begin{equation}
v\leq w\Longleftrightarrow v^{-1}\leq w^{-1}\Longleftrightarrow uvu^{-1}\leq uwu^{-1}.
\end{equation}

Let  $u,v,w\in W$ be such that $u\leq v\leq w$ then 
$u^{-1}v\leq u^{-1}w$ and $u\leq uv^{-1}w\leq w$. 
\end{prop}
\subsection{Parabolic subgroups}\label{Parasubg}
\subsubsection{Basic definitions}
Let $J\subset S$, then the {\sl standard parabolic subgroup} $W_J$ is the subgroup generated by $J$.
If $s\in S$ we will also use the notation $W_{\langle s\rangle}$ for the parabolic subgroup associated with $S\setminus\{s\}$.
A {\sl parabolic subgroup} is any subgroup conjugate to some $W_J$. 


The parabolic subgroups have the form
\[
  P= \{ w\in W \; : \; w(x)=x \text{ for all } x\in E \}
\]
for some subspace $E \subset V$, moreover if 
 $P$ is a parabolic subgroup and  
\[
  \Fix( P )=  \{ x \in V \; :\; w(x)=x \text{ for all } w\in P  \}
\]
then $P$ is itself a reflection group in $O( \Fix(P)^\perp )$. 
In particular, to each $w\in W$ we can associate a parabolic subgroup $\Gamma(w)$:
\begin{equation}\label{Gamma}
   \Gamma(w) =  \{ v\in W \; : \; \Fix(w) \subset \Fix(v) \}.
\end{equation}

\subsubsection{Simple generators and roots}
Let $P\subset W$ be a parabolic subgroup. Then $P$ is itself a reflection group, with reflection set $T\cap P$, and its roots form a subset of those of $W$. A natural set of positive roots for $P$ is $\Pi(P) = \Pi \cap P$.  Accordingly, there is a unique set of simple roots $\Delta(P)$ (see Remark~\ref{positivetosimple} above) and a set of simple reflections $S(P) = r^{-1}( \Delta(P) )$. We have $\Delta(P) \subset \Pi$. Note that  $\Delta(P)\subset \Delta$ does not hold in general, this only happens for standard parabolic subgroups.

Since there are several definitions of these sets $\Delta(P)$ or $S(P)$ given in the literature, in order to apply results from various references, it is in order to check that they are all equivalent.

\begin{rema}
A set of positive roots is $\Delta(P)$ for some parabolic subgroup $P$ if and only if they are such that the scalar product of any pair is nonpositive.
\end{rema}

%
%

\begin{prop}
 The simple reflections of $P$ are the reflections $t\in T\cap P$ satisfying $\Inv_R(t)\cap P = \{t\}$.
\end{prop}

\begin{proof}
A reflection $t\in P \cap T$ is simple as an element of $P$, if and only if it has only one inversion as an element of $P$.
Since $\Inv_R(t)\cap P$ is the right inversion set of $t$ as an element of $P$, the result follows.
\end{proof}

\begin{prop}
The set $\Delta(P) \subset \Pi(P)$ is the unique simple system of $P$ such that 
the fundamental chamber of $P$ contains that of $W$.
\end{prop}

\begin{proof}
Since $\Pi(P) \subset \Pi$, the positive span of $\Pi(P)$ is included in that of $\Pi$.
Taking the dual cone reverses inclusion, so the fundamental chamber of $P$ contains that of $W$.
\end{proof}

\subsubsection{The Bruhat graph}
The Bruhat graph on the vertex set $W$
is defined by putting an oriented arrow $w\to v$ if $vw^{-1}\in T$ and $\ell(v)<\ell(w)$. The unoriented underlying graph  is the Cayley graph of $W$ with the reflections as generating set.

\begin{prop}[\cite{dyer}]\label{dyerprop} 
Let $(W,S)$ be a Coxeter system and $P$ a parabolic subgroup, then  the restriction of the Bruhat graph to  $P$ is the Bruhat graph of $P$ for its canonical generators.
\end{prop}

In other words, if $\leq_{B_P}$ is the Bruhat order on $P$ and if $v,w\in P$ and $vw^{-1}$ is a reflection then 
\[
  v\leq_Bw\Longleftrightarrow v\leq_{B_P} w.
\]
This implies in particular that for $v,w\in P$ one has 
\[
 v\leq_{B_P} w\Longrightarrow v \leq_B w.
\]
The converse implication does not hold in general, see (\ref{counter}) below.

\section{Noncrossing partitions}\label{noncrois_sec}
We refer to \cite{armstrong} for the general facts on this subject.
\subsection{Definition of noncrossing partitions}
Let $c$ be a standard Coxeter element, then 
the set of noncrossing partitions associated to $c$, denoted by $ NC(W,c)$, is the set of all $w\in W$ such that $w\leq c$.
In the case where $W$ is the symmetric group $S_{n+1}$ with the Coxeter generators $s_i=(i ,i+1)$ and $c=s_1\ldots s_{n}$ is the cycle $(1,2,3,\ldots, n ,n+1)$ one can associate to any $w\in NC(W,c)$ the partition given by its cycle decomposition. This coincides with  the classical notion of noncrossing partition as  defined by Kreweras, see  \cite{kreweras}, \cite{biane}. In the sequel we denote by $NC_n$ the set of classical noncrossing partitions.

Endowed with the order $\leq$ the set $NC(W,c)$ is a lattice \cite{bradywatt}. Since all Coxeter elements are conjugate, the isomorphism class of this lattice  structure does not depend on $c$.
The  map $w\mapsto \Gamma(w)$ which associates  a parabolic subgroup to a noncrossing partition is injective, moreover one has 
\begin{equation}
v\leq w\Longleftrightarrow \Gamma(v)\subset\Gamma(w)
\end{equation}
(see for example \cite[Section~5.1.3]{armstrong}). 
In particular, the codimension of the fixed subspace, which is also the absolute length (see 
(\ref{absolutelengthfixdim})) gives a rank function on $NC(W,c)$, which is a rank function for the lattice structure.

Parabolic subgroups can be considered as generalized set partitions so that the map $\Gamma$ gives a way to consider noncrossing partitions
as particular set partitions.
In the classical case of noncrossing partitions  of $[1,n]$ for all $i<j<k<l$, putting $t_1=(i,k),\, t_2=(j,l)$  neither product $t_1t_2$ or $t_2t_1$  belongs to $NC_n$ and this property characterizes the crossing of two transpositions with disjoint supports. This leads to the following definition of a property which will play an  important role later.
 
\begin{defi}{(Bessis \cite[Definition~2.1.1]{bessis}).}\label{noncrois}
Two reflections $t_1,t_2$ are called {\it $c$-noncrossing} if either $t_1t_2\leq c$ or $t_2t_1\leq c$. 
\end{defi}

\begin{rema}\label{onlycommutations} \ 
\begin{enumerate}
 \item Bessis uses the notation $\mathrel{\|_c}$ to denote this relation, however this conflicts with the use of the same notation by Reading \cite{reading} to denote another relation. We will use Reading's notation later on (see Definition~\ref{defclusterrel}) so we will just say that  two reflections are $c$-noncrossing when needed, without using a specific notation for this relation.
\item Observe that if two reflections $t_1,t_2$ do not commute then one cannot have both $t_1t_2\leq c$ and $t_2t_1\leq c$ (this follows from the fact that $\Gamma(t_1t_2)=\Gamma(t_2t_1)=\langle t_1,t_2\rangle$ and the injectivity of $w\mapsto \Gamma(w) $ on $NC(W,c)$).
\end{enumerate}
\end{rema}

One can define noncrossing partitions in the same way when $c$ is a (general) Coxeter element, i.e.~is conjugated to some standard Coxeter element.  Then any noncrossing partition $w\leq c$ is a Coxeter element of a parabolic subgroup, see \cite[Lemma~1.4.3]{bessis}.  However our results crucially depend on properties of standard Coxeter elements, therefore in the following we will only consider such Coxeter elements. 

\begin{prop}\label{reading1prodsimples}
 Let $w \in  NC(W,c)$, then there exists an indexing ${\mathfrak s}_1,\dots,{\mathfrak s}_k$ of the 
 simple generators of $\Gamma(w)$ such that
 \[
    w = {\mathfrak s}_1 \cdots {\mathfrak s}_k.
 \]
\end{prop}

\begin{proof}
 This follows from results of Reading~\cite{reading}.
 More precisely, Theorem~6.1 from \cite{reading} shows that
 $w$ is the product, in some order, of the so-called cover reflections of 
 a Coxeter-sortable element. By Lemma~3.1 of the same reference 
 these cover reflections are the simple generators of a parabolic subgroup. 
 
 Alternatively, this follows from the results of Brady and Watt~\cite{bradywatt2}. In the case of a bipartite Coxeter element, \cite[Proposition~5.1]{bradywatt2} gives a way to compute a valid ${\mathfrak s}_k$, reducing the problem to finding ${\mathfrak s}_1, \dots, {\mathfrak s}_{k-1}$, which can be done inductively. Using Lemma~\ref{scs_moves}, it remains only to see how the result is transferred from $c$ to $scs$ where $s$ is a left descent of $c$.  Let $w = {\mathfrak s}_1 \cdots {\mathfrak s}_k$ be the factorization of $w \in NC(W,c)$ as a product of simple generators.  If $s\notin\{ {\mathfrak s}_1 , \dots , {\mathfrak s}_k \}$, we have $sws = (s{\mathfrak s}_1s) \cdots (s{\mathfrak s}_ks) \leq scs$ and the factors form a simple system, otherwise we can assume $s={\mathfrak s}_1$ and we have $sws = {\mathfrak s}_2 \cdots {\mathfrak s}_k {\mathfrak s}_1 \leq scs$.  
\end{proof}

In other words, each $w\in  NC(W,c)$ is a standard Coxeter element of its own parabolic subgroup $\Gamma(w)$, considered as a reflection group.
Observe that this property actually depends only on $w$ and not on the standard Coxeter element $c$ such that $w\leq c$.

  From $ii)$ of Remark \ref{onlycommutations} we deduce
\begin{prop}\label{commCoxeter}

Let $c=t_1\ldots t_n=t'_1\ldots t'_n$ be two minimal factorizations of $c$ into reflections, where $t'_1,\ldots,t'_n$ is a permutation of $t_1,\ldots ,t_n$, then one can pass from one factorization to the other by a succession of transpositions of neighbourhing commuting reflections $t_it_j\to t_jt_i$.
\end{prop}

\begin{proof} By induction on $n$, the length of $c$. If $t_1=t'_1$ then use the induction hypothesis for the Coxeter element $t_1c\in \Gamma(t_1c)$. If $t_1=t'_j$ with $j>1$ then $t_1t'_i\leq c$ and $t'_it_1\leq c$ for all $i<j$ therefore $t_1$ commutes with all $t'_i$ with $i<j$ and we can move it to the left to reduce to the preceding case.
\end{proof}
\subsection{Generalized Catalan and Narayana numbers} 
The number of elements of $ NC(W,c)$ is the generalized Catalan number $\Cat(W)$.  If $(W,S)$ is an irreducible system one has 
\[
 \Cat(W)=\prod_{i=1}^n\frac{h+e_i+1}{e_i+1}
\]
where the $e_i$ are the exponents of $W$ and $h$ is the Coxeter number, i.e.~the order of $c$ as a group element.  Analogously the number of elements of $ NC(W,c)$ of rank $k$ is the generalized Narayana number $\Nar_k(W)$, see e.g.~\cite{armstrong}.

Let $ NCF(W,c) \subset NC(W,c)$ denote the set of noncrossing partitions with full support.
Its cardinality is the {\it positive Catalan number} $\Cat^+(W)$, given by
\[
 \Cat^+(W)=\prod_{i=1}^n\frac{h+e_i-1}{e_i+1}.
\]
The numbers $\Cat(W)$ and $\Cat^+(W)$ are related by inclusion-exclusion.

The {\it positive Narayana number} $\Nar^+_k(W)$ is the number of noncrossing partitions of rank $k$ with full support. 
These numbers are related with Narayana numbers by inclusion-exclusion.

Define  the {\it Fuß-Catalan numbers} $\Cat^{(k)}(W)$ as
\begin{equation} \label{fusscat}
   \Cat^{(k)}(W) = \prod_{i=1}^n \frac{kh+e_i+1}{e_i+1}.
\end{equation}
This is the number of chains $w_1\leq \dots \leq w_k$ in $ NC(W,q)$, see \cite{chapoton1}.  Note that this is a polynomial in $k$.  It can be seen as a rescaling of the polynomial
\[
   \sum_{w\in W} t^{\ell_T(w)} = \prod_{i=1}^n (1+te_i).
\]

\subsection{The Kreweras complement}
For any $w\in NC(W,c)$ one defines its Kreweras complement as $K(w)=w^{-1}c$. The map $K$ is bijective, and defines an anti-automorphism of the lattice structure on $NC(W,c)$. The map $K$ is not involutive, rather $K^2(w)=c^{-1}wc$ is an automorphism of $NC(W,c),\leq)$. 
The  inverse anti-automorphism of $K$ is $K^{-1}(w)= c w^{-1} $. 
We sometimes denote $K_c$ to mark the dependence on $c$.

The Kreweras complement has a remarkable compatibility with respect to the Bruhat order, which will play a crucial role in this paper.

\begin{prop} \label{krew1}\

\begin{enumerate}
\item Let $v,w \in NC(W,c)$ with $v\lessdot w$, then we have:
\[
  v >_B w \Longrightarrow K(v) >_B K(w).
\]
\item The converse result holds under a supplementary hypothesis:
      let $v,w \in NC(W,c)$ with $v\lessdot w$, and suppose $v$ has full support in $ NC(W,c)$, then we have:
\[
  K(v) >_B K(w) \Longrightarrow v >_B w.
\]
\end{enumerate}
\end{prop}

Before the proof we need a lemma.

\begin{lemm}\label{krewlemm}
Let $w\in NC(W,c)$ and let $t$ be a reflection such that $tw<w$ and $tc<_Bc$, then $tw<_Bw$.
The same holds with right multiplication: if $wt<w$ and $ct<_Bc$, then $wt<_Bw$.
\end{lemm}

\begin{proof} 
Since $tc<_Bc$ it follows that $tc$ is a subword of length $n-1$ of a reduced expression for $c$, therefore $tc\in W_{\langle s\rangle}$ for some $s\in S$.

Next, we have $t \notin W_{\langle s\rangle}$.  Indeed $s$ is in the support of $t$, otherwise the subgroup $W_{\langle s\rangle}$ would contain both $t$ and $tc$, hence their product $c$, but clearly $c\notin W_{\langle s\rangle}$. 

As $t<w\leq c$, by Proposition~\ref{propelem} one has  $tw\leq tc$, therefore $tw\in W_{\langle s\rangle}$.  One has 
$w=t(tw)$ with $t \notin W_{\langle s\rangle}$ and $tw\in W_{\langle s\rangle}$, therefore $w\notin W_{\langle s\rangle}$.  

So $s$ belongs to the support of $w$ but not that of $tw$. It follows that $w$ cannot be obtained as a reduced subword of $tw$.  Knowing that either $tw <_B w$ or $w <_B tw$ holds, we get $tw <_B w$ as announced.

The case of right multiplication is analogous.
\end{proof}

\begin{proof}[Proof of Proposition~\ref{krew1}]

We first prove $i)$.  Suppose that $v$ and $w$ are such that $v \lessdot w$, and $v >_B w$.  Let $t\in T$ be the reflection such that $v=tw$.  

Note that $tw >_B w$ implies $w^{-1}t >_B w^{-1}$, which means that $t$ is not a right inversion of $w^{-1}$. So $w^{-1}(r(t)) \in \Pi $ by Proposition~\ref{inv_rootimages}.

By Lemma~\ref{krewlemm} and the assumptions on $v$ and $w$, we have $tc>_Bc$.  Consequently, we have $c^{-1}t>_B c^{-1}$, which means that $t$ is not a right inversion of $c^{-1}$, so $r(c^{-1}tc)=c^{-1}(r(t))$ by Proposition~\ref{inv_rootimages}.  Applying $K(w)=w^{-1}c$ on both sides, we get $K(w)(r(c^{-1}tc))=w^{-1}(r(t))$.  But as showned above, $w^{-1}(r(t)) \in \Pi $, and it follows that $c^{-1}tc$ is not a right inversion for $K(w)$.

We thus get $K(w)c^{-1}tc>_B K(w)$. Since $v=tw$, after simplification this gives $K(v) >_B K(w)$.

Now we prove $ii)$. Let $v=tw<w\leq c$ for a reflection $t$. If $tc<_Bc$ then $tc\in W_{\langle s\rangle}$ for some $s\in S$ and $tc<c$ therefore $v<tc\in W_{\langle s\rangle} $. We deduce from this  that, if  $v$ has full support, then $tc>_Bc$.  If $K(w)<_BK(v)$ then, again by Proposition 
\ref{inv_rootimages}  one has $r(c^{-1}tc)=c^{-1}(r(t))$  and $w^{-1}(r(t))=K(w)(r(c^{-1}tc))\in\Pi$ therefore $t$ is not a left inversion for $w$ and $tw>_Bw$.
\end{proof}
\subsection{An involutive automorphism for bipartite Coxeter elements}

In this section we consider the particular case of bipartite Coxeter elements.  Consider a partition $S = S_+ \uplus S_-$, where $S_+$ (respectively $S_-$) contains pairwise commuting elements.   Then let $c_+ = \prod_{s\in S^+}$ and $c_- = \prod_{s\in S^-}$. The bipartite standard Coxeter element is $c=c_+c_-$.
 
\begin{prop}\label{Ldef} The map $L$ defined by 
\[  
  L(w)=c_+wc_-
\]
is an involutive anti-automorphism of the poset $NC(W,c)$. 
\end{prop}
\begin{proof} 
The map $L$ is the composition of the maps $w\to w^{-1}$ which is an isomorphism from $NC(W,c)$ to $NC(W,c^{-1})$,
the isomorphism $w\mapsto c_+ wc_+$ from $NC(W,c^{-1})$ to $NC(W,c)$ and the Kreweras map $K$ which is an antiautomorphism of $NC(W,c)$, therefore $L$ is an anti-automorphism.  The fact that it is an involution follows from $c_\pm^2=e$. 
\end{proof}
We call the map $L$ the {\sl bipartite complement} on $NC(W,c)$. 

Like the Kreweras complement, the bipartite  complement  is  compatible  with the Bruhat order.

\begin{prop} \label{bipcomp1}\
\begin{enumerate}
\item Let $v,w \in NC(W,c)$ with $v\lessdot w$, then we have:
\[
  v >_B w \Longrightarrow L(v) >_B L(w).
\]
\item Let $v,w \in NC(W,c)$ with $v\lessdot w$, and suppose  that $v$ has full support, then we have:
\[
  L(v) >_B L(w) \Longrightarrow v >_B w.
\]
\end{enumerate}
\end{prop}

\begin{proof}
We first prove $i)$.
Suppose that $v=tw$ for some reflection $t$, that $v<w$ and $w<_Bv$, then $tc>_Bc$ by Lemma~\ref{krewlemm}. It follows, applying repeatedly Proposition~\ref{inv_rootimages}, that $t\notin S_+$ and $r(c_+tc_+)=c_+(r(t))$. Applying the same reasoning to $v=w(w^{-1}tw)$ gives that $c(w^{-1}tw)>_Bc$ and $w^{-1}tw\notin S_-$ therefore, since $w^{-1}(r(t))=r(w^{-1}tw)$, one has  $r(c_w^{-1}-twc_-)=c_-w^{-1}(r(t))$.
One has $L(v)=c_+tc_+L(w)$ and  $L(w)^{-1}(r(c_+tc_+)=r(c_w^{-1}-twc_-)$ and $c_+tc_+$ is not a left inversion of $L(w)$, thus $L(v)>_BL(w)$.

We now prove $ii)$. Let $v=tw<w$ for some reflection $t$ and suppose that  $v$ has full support. Arguing as in to the proof of $ii)$ in Proposition~\ref{krew1} we get  $tc>_Bc$ moreover let $t'=v^{-1}tv$ then $v=wt'$ and, similarly, $ct'>_Bc$. In particular, $t$ is not a left inversion of  $c_+$ and $t'$ is not a right inversion of $c-$.  Suppose now that $Lv>_BLw$ i.e. $c_+vc_->_Bc_+tvc_-$ then  $c_+tc_+$ is a left inversion of $c_+vc_-$ and, since $c_+(r(t))$ is positive we see that  $c_-v^{-1}(r(t))$ is a negative root. Since $t'$ is not a right inversion  of $c_-$ we see that $v^{-1}(r(t))$ is a negative root therefore $w=tv<_Bv$.
\end{proof}

\subsection{Interval partitions}\label{interval_sec}
They form a natural subset of $ NC(W,c)$ defined as follows.

\begin{defi}
Given a Coxeter element $c$, an {\sl interval partition} is an element of $NC(W,c)$ whose associated parabolic group is standard, i.e. the
interval partitions are the $w\in NC(W,c)$ such that $\Gamma(w) = W_J$ for some $J\subset S$. 
We denote by $ INT(W,c)$ the set of interval partitions of $NC(W,c)$.
\end{defi}

Equivalently, $w\in W$ is an interval partition if and only if $w\leq_B c$.  Since a reduced word for $c$ contains each simple reflection exactly once, it is easy to see that each subword is a reduced word, therefore the orders $\leq$ and $\leq_B$ coincide on $ INT(W,c)$ and give it the structure of a Boolean lattice, isomorphic to the lattice of subwords of a reduced expression for $c$, or to the lattice of subsets of $\Delta$, or of $S$.

One can check that, when $W$ is the symmetric group with its canonical Coxeter generators and $c$ is the cycle $(1\ldots n)$,
the interval partitions coincide with the classical ones (cf.~e.g.~\cite{speicher}).

\begin{defi}
For $w\in NC(W,c)$ we denote by $\underline w$ the largest interval partition below $w$ and by $\overline w$ the smallest interval partition above $w$ in $( NC(W,c),\leq)$.

 The {\it relative Kreweras complement} of $w$ with respect to $\overline w$ is $K(w,\overline w) =  w^{-1}  \overline w$.
\end{defi}
Since the restriction to $INT(W,c)$ of the abolute order is a lattice order the existence and uniqueness of $\overline w$ and $\underline w$ is immediate. 
One can also characterize $\underline w$ by its associated simple system:
\[
  \Delta(\Gamma(\underline w))=\Delta(\Gamma(w))\cap \Delta(W).
\]
Analogously 
if  $J\subset S$ is the support of $w$, then the subword $\gamma$ of $c$ consisting of elements of $J$ is the unique Coxeter element in  $NC(W,c)\cap W_J$ and one has $\gamma=\overline w$,  moreover $w$ has full support in $\Gamma(\gamma)=W_J$. It is easy to see that the map $w\mapsto \overline w$ is a lattice homomorphism. 

There is also a characterization of $\overline w$ as a maximum, which we leave to the reader to check.

\begin{prop}  \label{conncomp0}
For $w\in  NC(W,c)$ one has
\[
  \overline w=\max\{\eta\in  INT(W,c)\,|\,\eta\leq_B w\}.
\]  
\end{prop}

We now consider the upper ideal $\{w\in NC(W,c)\,|\, w\geq\gamma\}$ for certain interval partitions $\gamma$.

\begin{prop}\label{leftdesc}
Let $s\in  S$ be a left descent of $c$ (i.e. $sc<_Bc$).

\begin{enumerate}
 \item Let $w\in NC(W,c)$ be such that $s\leq w$, then $sw\in W_{\langle s\rangle}$.
 \item The map $v\mapsto sv$ defines a lattice homomorphism from $NC( W_{\langle s\rangle},sc)$ to $NC(W,c)$ whose image is the subset 
$\{w\in NC(W,c)|w\geq s\}$.
\end{enumerate}
There is a similar statement with right descents.
\end{prop}

\begin{proof}
$i)$ It follows from Proposition \ref{propelem} that $sw\leq sc$ hence $sw\in NC(W_{\langle s\rangle})$.

$ii)$ It is easy to see that if $v\in W_{\langle s\rangle}$ then $\ell_T(sv)=\ell_T(v)+1$  which implies the result, using Proposition \ref{propelem} to see that it is a homomorphism and the first statement to see that the map is surjective.

Finally the case of right descents follows by considering the order preserving isomorphism  $w\mapsto w^{-1}$ from $NC(W,c)$ to $NC(W,c^{-1})$.
\end{proof}

\begin{coro}\label{leftdesccor}
Let $a$ be an initial subword of $c$ and $b$ be a final subword with $\ell(a)+\ell(b)\leq \ell(c)$, then $ab \in  INT(W,c)$ and 
the set $\{w\in NC(W,c)|ab\leq w\}$ is the image of the map  $v\mapsto avb$ from $NC(W_{\langle ab\rangle},a^{-1}cb^{-1})$ to $NC(W,c)$, which is a lattice homomorphism. 
\end{coro}
\begin{proof} By induction on $\ell(a)+\ell(b)$ using the preceding proposition.
\end{proof}
In the bipartite case, one can say more.
\begin{prop}\label{leftdesccorbip}
Let $c$ be a bipartite Coxeter element then the complement map $L$ restricts to an involution on $INT(W,c)$.
\end{prop}
\begin{proof} It is easy to see that $L$ corresponds to the complementation map if on identitifes $INT(W,c)$ with the set of subsets of $\Delta$ or $S$.
\end{proof}
Since $L$ is an involutive anti-automorphism one can see that, if $\gamma\in INT(W,c)$ then
$L$ exchanges the sets
$\{w |\overline w=\gamma\}$ and $\{w |\underline w=L(\gamma)\}$.
\section{\texorpdfstring{Two order relations on $W$}{Two order relations on W}}\label{order_sec}

\subsection{Definition of the orders}
Consider a Coxeter system $(W,S)$.
For any covering relation of the absolute order $v\lessdot w$ in $W$ one has $vw^{-1}\in T$, therefore $v$ and $w$ are comparable for the Bruhat order, namely one has $v<_Bw$ if $\ell(v)<\ell(w)$ or $w<_Bv$ if $\ell(w)<\ell(v)$. 
We introduce two order relations which encode this distinction by refining the absolute order.

\begin{defi}
For any covering relation $v\lessdot w$ in $W$ define 
\begin{align*}
v\sqsubsetdot w \ &\text{ if }\ v<_Bw, \\
v\lldot w \ &\text{ if }\ w<_Bv,
\end{align*}
and extend these  relations to order relations $\sqsubset$ and $\ll$  on $W$ by transitive closure. 
\end{defi}

Since $\sqsubsetdot$ 
and $\lldot$ are included in $\lessdot$ (seeing relations as sets of pairs), they are acyclic and their transitive
closures $\sqsubset$ and $\ll$ are partial orders which are included in $\leq$ and whose cover relations are $\sqsubsetdot$ and $\lldot$.
This justifies the previous definition. Also the rank function for $\leq$ serves as a rank function for $\sqsubset$ and $\ll$.

The order $\sqsubset$ was introduced by the second author in \cite{josuat} on the set $ NC(W,c)$, while $\ll$ was introduced by Belinschi and Nica in 
\cite{belinschinica} in the case of classical noncrossing partitions, both with different definitions. We shall see below that they are specializations of our definitions.

Observe that, as an immediate consequence of the definitions, one has 
\begin{align*}
   v\sqsubset w \quad &\Longrightarrow \quad v\leq w, \ v\leq_B w, \\
   v\ll w  \quad  &\Longrightarrow\quad v\leq w,\  w\leq_Bv.
\end{align*}
The opposite implications, however, do not hold in general. 
 Here is a  counterexample
in the symmetric group $\mathfrak{S}_5$: take (in cycle notation)
\begin{equation}\label{counter}
   v = (2,4), \qquad w=(1,5)(2,3,4).
\end{equation}
We have $v\leq w$, $\ell(v)=3$ and $\ell(w)=9$. There are two elements between $v$ and $w$ in the absolute order, which are
$(1,5)(2,4)$ and $(2,3,4)$. Their respective Coxeter lengths are $10$ and $2$, therefore we have:
\begin{align*}
   v &= (2,4) \sqsubsetdot (1,5)(2,4) \lldot (1,5)(2,3,4)=w, \\
   v &= (2,4) \lldot (2,3,4) \sqsubsetdot (1,5)(2,3,4)=w.
\end{align*}
It follows that $v \nsqsubset w$. 
On the other side, we have $v\leq_B w$ and a shortest path in the Bruhat graph is (an arrow with label $t$
represents multiplication by $t$ on the right):
\[
  (1,5)(2,3,4) \overset{(1,3)}{\longrightarrow} (1,4,2,3,5)
               \overset{(2,5)}{\longrightarrow} (1,4,2)(3,5)
               \overset{(3,5)}{\longrightarrow} (1,4,2)
               \overset{(1,2)}{\longrightarrow} (2,4).               
\]
The respective lengths of the elements in this path are $9$, $8$, $5$, $4$, $3$. Note that  $v$ and $w$ are both noncrossing partitions but the path goes through elements that are not 
noncrossing partitions.


Let $w\in W$ and let $\Gamma(w)$ be the associated parabolic subgroup. Since the orders $\sqsubset$ and $\ll$ can be defined using only the Bruhat graph and the absolute order, it follows 
from Proposition~\ref{dyerprop} that, on the set $\{v\in W\,|\,v\leq w\}$ the orders $\sqsubset$ and $\ll$ depend only on the  Bruhat order on $\Gamma(w)$ (with its canonical system of generators).

In the sequel we shall mainly be interested in the restrictions of these order relations to $ NC(W,c)$ for some standard Coxeter element $c$.

\subsection{The Kreweras complement }
Using the orders $\sqsubset$ and $\ll$ we can reformulate Proposition  \ref{krew1}.

\begin{prop} \label{krew1bis} \ 
Let $v,w \in NC(W,c)$ with $v\lessdot w$.
\begin{enumerate}
\item
We have:
\[
  v \ll w \Longrightarrow K(w) \sqsubset K(v).
\]
\item
Suppose $v$ has full support in $ NC(W,c)$, then we have:
\begin{align} \label{eq_krew2a}
  K(w) \sqsubset K(v) \Longrightarrow v \ll w.
\end{align}
\end{enumerate}
\end{prop}


Note that by taking the contraposition in $i)$ of Proposition~\ref{krew1bis}, we see that the same result holds with $K^{-1}$ instead of $K$.
The situation is slightly different for $ii)$ and we need another argument. The map $\iota: w \mapsto w^{-1}$ leaves $S$ and $T$ invariant, preserves length, absolute length and support, moreover it sends $ NC(W,c)$ to $ NC(W,c^{-1})$. We have 
\[ 
  (\iota \circ K_c)(w) = \iota( w^{-1} c) = c^{-1} w = c^{-1} \iota(w^{-1}) = ( K_{c^{-1}}^{-1} \circ \iota ) (w).
\]
Applying the map $\iota$ to Equation~\eqref{eq_krew2a} shows that $ii)$ also holds with $K^{-1}$ instead of $K$. Using these remarks we can rephrase the proposition
as follows.

\begin{prop}
Let $v,w \in NC(W,c)$ with $v\lessdot w$. Then at least one of the relations $v\sqsubsetdot w$ or $K(w)\sqsubsetdot K(v)$ holds.
If both hold, neither $v$ nor $K(w)$ have full support.
\end{prop}

There are similar statements for the map $L$ in the bipartite case, which we leave to the reader.
\subsection{\texorpdfstring{Characterization and properties of the order $\sqsubset$ on $ NC(W,c)$}{The order [}}
We consider a Coxeter system $(W,S)$ and a standard Coxeter element $c$.
\begin{prop}\label{sqchar}
Let $v,w\in   NC(W,c)$ and denote by $\leq_{B_w}$ the Bruhat order on $\Gamma(w)$, then the three properties below are equivalent:
\begin{enumerate}
 \item $v\sqsubset w$, 
 \item $v\leq w$ and $v\leq_{B_w} w$,
 \item $\Delta(\Gamma(v)) \subset \Delta(\Gamma(w))$.
\end{enumerate}
\end{prop}
\begin{proof}

If $v\sqsubset w$ then, as we saw, $v\leq w$ and $v\leq_{B_w}w$ so that $i)$ implies $ii)$.

Suppose now that $ii)$ holds.
According to Proposition~\ref{reading1prodsimples}, write $w$ as a product of the simple reflections of $\Gamma(w)$, say $w={\mathfrak s}_1\cdots {\mathfrak s}_k$. 
Any $v\leq_{B_w}w$ is the product of a subword ${\mathfrak s}_{i_1}\dots,{\mathfrak s}_{i_l}$ and $\Delta(\Gamma(v))=\{{\mathfrak s}_{i_1},\dots,{\mathfrak s}_{i_l}\}\subset \Delta(\Gamma(w))$ so that $ii)$ implies $iii)$. 

Finally if $\Delta(\Gamma(v)) \subset \Delta(\Gamma(w))$ then by taking the generated groups
one has $\Gamma(v) \subset \Gamma(w)$ and $v\leq w$.  If $\ell_T(v)=\ell_T(w)-1$ then $vw^{-1}\in T$ therefore either $v<_{B_w}w$ or $w<_{B_w}v$. Since 
$\Delta(\Gamma(v))\subset\Delta(\Gamma(w))$ but $\Delta(\Gamma(v))\neq\Delta(\Gamma(w))$  there exists a simple reflection 
${\mathfrak s}\in \Delta(\Gamma(w))\setminus\Delta(\Gamma(v))$. As $v$ is a product of reflections in $\Delta(\Gamma(v))$ it follows that $w$ cannot be obtained as   a subword of  any reduced decomposition of $v$, therefore  $v<_{B_w}w$ and $v$ is obtained by taking a subword of length $k-1$ of the reduced decomposition $w={\mathfrak s}_1\cdots {\mathfrak s}_k$. The proof now follows by induction on the cardinality of $\Delta(\Gamma(w))\setminus\Delta(\Gamma(v))$, using the fact that the Bruhat graph of the subgroups $\Gamma(v)$ are obtained by restriction of the Bruhat graph of $W$ (Proposition~\ref{dyerprop}).
\end{proof}

Property $iii)$ was the original definition of the order $\sqsubset$ on $ NC(W,c)$ in \cite{josuat}.
Since any subset of $\Delta(\Gamma(w))$ is a simple system in $\Gamma(w)$ we see that  the lower ideal $\{ v \; : \; v\sqsubset w \}$ is isomorphic
to the boolean lattice of order $\rk(w)$, in particular its cardinality is $2^{\rk(w)}$, moreover on this ideal the three order relations $\leq,\leq_{B_w}$ and $\sqsubset$ coincide.

Note that $\Delta(\Gamma(w))\subset\Delta(\Gamma(v))$ if and only if $S(\Gamma(w))\subset S(\Gamma(v))$ using the  bijection between positive roots and reflections. 

Let now $c$ be a standard Coxeter element, 
since $\Gamma(c)=W$ and $\Delta(W)=S$ it follows that 
the  interval partitions in $INT(W,c)$ are exactly the noncrossing partitions $w\in NC(W,c)$ satisfying $w \sqsubset c$.

\begin{prop}
Let $w\in   NC(W,c)$ and $\gamma\in  INT(W,c)$ then 
\[
  \gamma\leq w\Longleftrightarrow \gamma\sqsubset w.
\]
\end{prop}
\begin{proof}
If $\gamma\sqsubset w$ then obviously $\gamma\leq w$. If $\gamma\leq w$ then $\Gamma(\gamma)\subset\Gamma(w)$ therefore 
$S(\Gamma(\gamma))\subset \Gamma(w)$. Since the elements of 
$S(\Gamma(\gamma))$ are simple reflections in $W$ they are also simple reflections in $\Gamma(w)$ therefore 
$S(\Gamma(\gamma))\subset S(\Gamma(w))$ and $\gamma\sqsubset w$ by $iii)$ of Proposition \ref{sqchar}.
\end{proof}

As is clear from its definition the order $\sqsubset$ is not invariant under conjugation and   one cannot, as in the case of the abolute order, reduce the study of $\sqsubset$ on $NC(W,c)$ to the case where the Coxeter element $c$ is bipartite. Actually we will see in Section \ref{examples_sec} that, already for $S_4$, the order type of the poset $(NC(W,c),\sqsubset)$ does depend on the standard Coxeter element $c$.
\subsection{A simplification property}

The following result is similar to Proposition~\ref{propelem}.

\begin{lemm}\label{simplification}
Let $t\in T$ and $u,v\in  NC(W,c)$ be such that $u\leq v\sqsubset vt\leq c$  then $u\sqsubset ut\leq vt$.
\end{lemm}

\begin{proof}
First one has $u\leq ut\leq vt$ by Proposition~\ref{propelem}.
Observe that the statement of the lemma depends only on the
absolute and the Bruhat orders  inside the parabolic subgroup $\Gamma(vt)$
therefore we can assume that $vt=c$ is a standard Coxeter element.
Since $ct\sqsubset c$ there exists
$s\in S$ such that $ct$ is a standard Coxeter element in $W_{\langle
s\rangle}$. It follows that $u\in W_{\langle s\rangle}$ and $t\notin W_{\langle s\rangle}$ therefore $ut\notin W_{\langle s\rangle}$ and $ut$ cannot be obtained as a subword of a reduced expression of $u$, thus $u\leq_B ut$ and $u\sqsubset ut$.
\end{proof}

\subsection{\texorpdfstring{$( NC(W,c),\sqsubset)$ as a flag simplicial complex}{NC(W,c) as a flag simplical complex}}

We now define a symmetric binary relation on the set $T$ of reflections such that $( NC(W,c),\sqsubset)$ is  the face poset of the associated flag simplicial complex. 
Recall that a simplicial complex is called a {\it flag} simplicial complex if, given vertices $v_1,\dots,v_k$, the set $\{v_1,\dots,v_k\}$ is a face
of the complex if and only if for all $1\leq i<j\leq n$, the set $\{v_i,v_j\}$ is a face. In order to define such a complex  it is enough to a 
give the symmetric relation on the vertices: $\{v_i,v_j\}$  is a face.

\begin{defi} \label{defi_Xi}
Let $\mathrel {\between_c} $ be the binary relation on $T$ such that  $t\mathrel {\between_c} u$ if  and only if
\begin{itemize}
 \item $t,u$ are $c$-noncrossing (see Definition~\ref{noncrois}),
 \item $\langle r(t) | r(u) \rangle \leq 0$.
\end{itemize}
Let $\Xi(W,c)$ denote the flag simplicial complex with vertex set $T$ associated to the relation 
$\mathrel {\between_c} $.
\end{defi}

\begin{theo}
The map $w \mapsto S( \Gamma(w) )$ is a bijection from $ NC(W,c)$  onto $\Xi(W,c)$.  Moreover, for $v,w\in NC(W,c)$ we have $v\sqsubset w$ if and only if  $S( \Gamma(v) ) \subset  S( \Gamma(w) )$.
\end{theo}

\begin{proof}
The image of $w$ is a subset of $T$ and taking the subgroup  generated  by this subset gives $\Gamma(w)$.
Since the map $\Gamma$, restricted to $ NC(W,c)$, is injective it follows  $w \mapsto S( \Gamma(w) )$ is also injective.
It remains to show that its image is $\Xi(W,c)$.

Since $S(\Gamma(w))$ is a simple system, the scalar product of any two of its elements is nonpositive.
Moreover by Proposition~\ref{reading1prodsimples}, there is an indexing $S(\Gamma(w)) = \{{\mathfrak s}_1,\dots,{\mathfrak s}_k\}$
such that ${\mathfrak s}_1\cdots {\mathfrak s}_k \leq c$,  consequently ${\mathfrak s}_i{\mathfrak s}_j \leq c$ if $i<j$.
It follows that  ${\mathfrak s}_i \between_c {\mathfrak s}_j$ for all $i<j$, therefore $S(\Gamma(w)) \in \Xi(W,c)$.

It remains to show that the map is surjective. Let $X \in \Xi(W,c)$.
Because $\langle r(t) | r(u) \rangle \leq 0$ for all $t,u\in X$, this is a simple system, i.e.~$X=S(P)$
for some parabolic subgroup $P \subset W$. Now consider a directed graph $G$ defined as follows:
\begin{itemize}
 \item its vertex set is $X$,
 \item there is an edge from $t\in X$ to $u\in X$ if $\langle r(t) | r(u) \rangle < 0$ and $tu\leq c$.
\end{itemize}
It is well defined because, $\langle r(t) | r(u) \rangle < 0$ implies $tu\neq ut$ therefore, according to Remark~\ref{onlycommutations}, we cannot have
both $tu \leq c$ and $ut\leq c$.
The undirected version of $G$ is the Coxeter graph of $P$ (without labels on the edges). Since $P\subset W$ is finite, a classical argument implies that 
 $G$ is acyclic, so that it is possible to find an indexing $X=\{ {\mathfrak s}_1,\dots,{\mathfrak s}_k\}$ such that 
the existence of a directed edge ${\mathfrak s}_i \to {\mathfrak s}_j$ implies $i<j$. It follows that ${\mathfrak s}_i {\mathfrak s}_j \leq c$
for all $1\leq i<j\leq k$.

Let us consider the case $k = \# X = 3$, so suppose we have $X=\{{\mathfrak s}_1,{\mathfrak s}_2,{\mathfrak s}_3\}$ with ${\mathfrak s}_1 {\mathfrak s}_2 \leq c$, ${\mathfrak s}_1 {\mathfrak s}_3\leq c$, and ${\mathfrak s}_2 {\mathfrak s}_3 \leq c$. Since ${\mathfrak s}_2 {\mathfrak s}_3 \in  NC(W,c)$ and this poset is a ranked lattice, the least upper bound of ${\mathfrak s}_2$ and ${\mathfrak s}_3$ is ${\mathfrak s}_2 {\mathfrak s}_3$.
On the other hand, by Proposition~\ref{propelem}, we get ${\mathfrak s}_2, {\mathfrak s}_3 \leq {\mathfrak s}_1 c$ therefore ${\mathfrak s}_2 {\mathfrak s}_3 \leq {\mathfrak s}_1c$, and consequently ${\mathfrak s}_1 {\mathfrak s}_2 {\mathfrak s}_3\leq c$, so that $X = S(\Gamma({\mathfrak s}_1 {\mathfrak s}_2 {\mathfrak s}_3))$. The general case follows by  induction on $\# X$ and this shows the surjectivity.

Eventually, the fact that the bijection preserves the order follows from Proposition~\ref{sqchar}.
\end{proof}

\begin{rema}
 The complex $\Xi(W,c)$ also appears in the work of Reading \cite{reading,reading2}.  In \cite{reading2}, he defines the {\it canonical join complex} of a join-semidistributive lattice.  In the case of the $c$-Cambrian lattice, this complex is precisely our $\Xi(W,c)$. This is essentially shown in \cite[Section~6]{reading}.  We thank Henri M\"uhle for his careful explanations about this.
\end{rema}

%
%
%

%

\subsection{\texorpdfstring{Chains in $NC(W,c)$}{Chains in NC(W,c)}}

A factorization $c = t_1 \dots t_{n}$ where each $t_i$ is a reflection is minimal.  The number of such minimal factorizations of the cycle $c$ is (Deligne's formula):
\[
   \frac{n! h^n }{|W|}.
\]
One can interpret this result as the counting of maximal chains in $ NC(W,c)$.  This number can be obtained from the leading term in $k$ of $\Cat^{(k)}(W)$ and the identity $\prod_i (e_i+1) = |W|$.

A refined enumeration of maximal chains in $NC(W,c)$ using the relation $\sqsubsetdot$ was obtained in \cite{josuat}.

\begin{defi}[\cite{josuat}]
 For each maximal chain $\varpi = (w_i)_{0\leq i\leq n}$ in $NC(W,c)$ (i.e.~we have $\rk(w_i)=i$ and $w_i\lessdot w_{i+1}$),
 we define $\nir(\varpi)$ as the number of $i\in\{0,\dots,n-1\}$ such that $w_i\lldot w_{i+1}$, and
\[
   M(W,q) = \sum q^{\nir(\varpi)}
\]
where we sum over $\varpi$ maximal chain in $NC(W,c)$.
\end{defi}
The second author showed in \cite{josuat} that this polynomial is a rescaled version of {\it Fuß-Catalan numbers} $\Cat^{(k)}(W)$:
\begin{align} \label{maxchain_fusscat}
  M(W,q) = n! (1-q)^n \Cat^{(\frac{q}{1-q})}(W).
\end{align}
Note that the inverse relation is
\[
   \Cat^{(k)}(W) = \frac{1}{n!} (1+k)^n M\big(W,\tfrac{k}{1+k}\big)
\]
and it follows from \eqref{maxchain_fusscat} and \eqref{fusscat} that there exists a formula in terms of the degrees of the group:
\begin{align} \label{enum_maxchains}
  M(W,q) = \frac{n!}{|W|} \prod_{i=1}^n \big( d_i + q (h-d_i) \big).
\end{align}
We use here the degrees $d_i=e_i+1$ rather than the exponents because the formula is more compact.  
This relation will be used in Section~\ref{sec_fullrefs}.

Much more can be said in the case of the symmetric group \cite{bianejosuat},
%
%
%
it would be interesting to investigate the existence of similar formulas for other types.

\subsection{\texorpdfstring{Characterization and properties of the order $\ll$}{The order <<}}
\label{char_order_ll_sec}

The following proposition is the analog, for the order $\ll$, of Proposition \ref{sqchar}.

\begin{prop} \label{llfull}
Let $v,w\in   NC(W,c)$ and denote by $\leq_{B_w}$ the Bruhat order on $\Gamma(w)$, then the three properties below are equivalent:
\begin{enumerate}
 \item $v\ll w$,
 \item $v\leq w$ and $w\leq_{B_w}v$,
 \item $v\in \Gamma(w)$ and $v$ has full support as an element of $\Gamma(w)$.
\end{enumerate}
\end{prop}


\begin{proof}
We have already noted  that $i)$ implies $ii)$. 

If $ii)$ holds then $w$ can be obtained as a subword of a reduced decomposition of $v$ in $\Gamma(w)$. Since $w$ has full support in $\Gamma(w)$ it follows that $v$ has full support in $\Gamma(w)$ and $iii)$ holds.

Suppose now that $iii)$ holds, thus that $v$ has full support in $\Gamma(w)$.
If $\ell_T(v)=\ell_T(w)-1$ then either $w\leq_{B_w}v$ or $v\leq_{B_w}w$. Since $w$ is a Coxeter element in $\Gamma(w)$ (cf.~Proposition~\ref{reading1prodsimples})) if $v\leq_{B_w}w$ then $v$ cannot have full support. It follows that $w\leq_{B_w}v $ therefore $v\ll w$. If $\ell_T(v)=\ell_T(w)-l,\, l>1$
let $x=K(v)=wv^{-1}$ be the Kreweras complement of $v$ and let $x=\gamma_1\ldots \gamma_l$ be a reduced decomposition in $\Gamma(x)\subset \Gamma(w)$. The sequence  $x_0=e,x_1=\gamma_1,\ldots, x_l=x $ satifies $x_0=e\leq_{B_x} x_1\leq_{B_x}\ldots \leq_{B_x}x_l=x $ and the Bruhat graph of $\Gamma(x)$ is the restriction of the Bruhat graph of $\Gamma(w)$ therefore $x_0=e\leq_{B_w} x_1\leq_{B_w}\ldots \leq_{B_w}x_l=x $. Applying $ii)$ of Proposition~\ref{krew1bis} and using induction on $i$, 
one has $v\lldot x_{l-i}^{-1}w$ for all $i$.
\end{proof}

Let $ NC_n$ be the set of classical noncrossing partitions, corresponding to the symmetric group and the cycle $(1,\ldots,n)$ as Coxeter element,  then Belinschi and Nica~\cite{belinschinica} defined an order relation  $\ll$ on $ NC_n$  by $\pi \ll \rho $ if $\pi\leq\rho$ and for every block $B$ of $\rho$, the minimum and the maximum of $B$
belong to  the same block of $\pi$. It is easy to see that this agrees with our definition in this case.

\begin{prop}  \label{conncomp}
The poset $( NC(W,c),\ll)$ has $2^n$ connected components. They are the lower ideals
\begin{equation}  \label{conncompeq}  
   \{ w \in  NC(W,c) \; : \; w \ll \gamma \} = \{ w \in  NC(W,c) \; : \; \overline w = \gamma \}
\end{equation}  
where $\gamma$ runs through the $2^n$ interval partitions of $W$.
In particular, interval partitions are the maximal elements for the order $\ll$.
\end{prop}

\begin{proof}
 From the characterization of $\ll$ in Proposition~\ref{llfull}, we can see that the two sets on both sides of \eqref{conncompeq} are the same.
 Clearly, each lower ideal considered here is connected, so it remains to show that they cannot be connected with each other.
 Let $v,w\in NC(W,c)$ with $v\ll w$, we thus have to show $\overline v = \overline w$, i.e. $v$ and $w$ have the same support.
 Since $v\ll w$, we have $w\leq_Bv$ therefore  $\supp(w)\subset\supp(v)$. If there exists  $s\in \supp(v)\setminus\supp(w)$ then  $w\in W_{\langle s\rangle}$ and $\Gamma(w)\subset W_{\langle s\rangle}$ which contradicts $v\in\Gamma(v)\subset \Gamma( w)$. 
\end{proof}


Since $c$ is itself an interval partition, the poset $(NCF(W,c),\ll)$ is one of the connected components described in the previous proposition.
We can call it the {\it main connected component}. The other connected components are of similar nature, since they can be seen as the 
main connected component of a standard parabolic subgroup. It is therefore enough to study the properties of the main connected component.

%

\begin{prop} \label{booleanideals0}
Let $w\in  NC(W,c)$, then the relative Kreweras complement $v\to v^{-1} \overline w $ defines
a poset anti-isomorphism 
\begin{equation}
  \big( \{ v\in  NC(W,c) \, : \, v \gg w \} , \ll \big) 
  \longrightarrow 
  \big( \{ v \in  NC(W,c) \, : \, v \sqsubset w^{-1} \overline w \} , \sqsubset \big). 
\end{equation}
\end{prop}

\begin{proof}
 If $w$ has full support, i.e., $\overline w = c$, this  follows from $ii)$ of Proposition~\ref{krew1bis}.
 Otherwise, $w$ has full support as an element of the standard parabolic subgroup $\Gamma(\overline w)$.
 Note that we have
 \[
    \{ v\in  NC(W,c) \, : \, v \gg w \} \subset \Gamma(\overline w)
 \]
 by Proposition~\ref{conncomp}. Therefore we can apply Proposition~\ref{krew1bis} in the subgroup $\Gamma(\overline w)$
 which yields the result. 
\end{proof}

\begin{coro} \label{booleanideals}
Upper ideals for $\ll$ in $NC(W,c)$ are boolean posets.
\end{coro}

\begin{proof}
 It was noted above the lower ideals for $\sqsubset$ are boolean. The result follows from
 the anti-isomorphism of the previous proposition, since boolean posets are anti-isomorphic to themselves.
\end{proof}

\section{Examples}\label{examples_sec}

In this section we show a few pictures of the symmetric groups which illustrate the main properties of our objects.

\subsection{\texorpdfstring{The case of  $S_3$}{The case of S 3}}
Permutations are denoted by their nontrivial cycles, $e$ is the identity element.  The Cayley graph is in Figure~\ref{figcayleyS3}, the Bruhat graph is obtained by orienting the edges downwards. The graph of the relation $\between_c$ is in Figure~\ref{figbetween3}.

\begin{figure}[h!tp]
\begin{tikzpicture}[scale=.4]
\draw (0,-2) node{$(12)$};
\draw (0,4) node{$(123)$};
\draw (8,4) node{$(132)$};
\draw (8,-2) node{$(23)$};
\draw (4,-6) node{$e$};
\draw (4,8) node{$(13)$};
\draw (0,-1.5) --(0,3.5);
\draw (.5,-1.5) --(7.5,3.5);
\draw (7.5,-1.5) --(.5,3.5);
\draw (8,-1.5) --(8,3.5);
\draw (.5,-2.5) --(3.5,-5.5);
\draw (7.5,-2.5) --(4.5,-5.5);
\draw (0,4.5) --(3.5,7.5);
\draw  (8,4.5) --(4.5,7.5);
\draw  (4,-5.5) --(4,7.5);
\end{tikzpicture}  
\caption{The Cayley graph of $S_3$.\label{figcayleyS3}}
\end{figure}
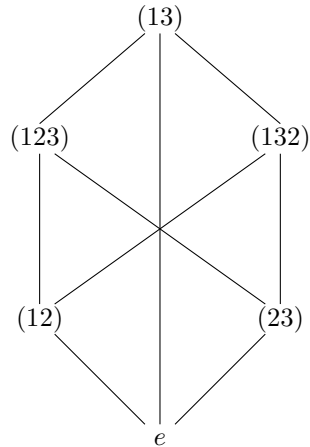

\bigskip

\begin{figure}
\begin{tikzpicture}[scale=.4]
\draw (0,-2) node{$(12)$};
\draw (0,4) node{$(123)$};
\draw (8,4) node{$(132)$};
\draw (8,-2) node{$(23)$};
\draw (4,-6) node{$e$};
\draw (4,-2) node{$(13)$};
\draw (0,-1.5) --(0,3.5);
\draw (.5,-1.5) --(7.5,3.5);
\draw (7.5,-1.5) --(.5,3.5);
\draw (8,-1.5) --(8,3.5);
\draw (.5,-2.5) --(3.5,-5.5);
\draw (7.5,-2.5) --(4.5,-5.5);
\draw [red,,dashed](0,3.5) --(3.5,-1.5);
\draw  [red,,dashed](8,3.5) --(4.5,-1.5);
\draw  (4,-5.5) --(4,-2.5);
\draw (14,-2) node{$(12)$};
\draw (18,2) node{$(123)$};
\draw (22,-2) node{$(23)$};
\draw (18,-6) node{$e$};
\draw (18,-2) node{$(13)$};
\draw (14.5,-1.5) --(17.5,1.5);
\draw (21.5,-1.5) --(18.5,1.5);
\draw (14.5,-2.5) --(17.5,-5.5);
\draw (21.5,-2.5) --(18.5,-5.5);
\draw (18,-5.5) --(18,-2.5);
\draw [red,dashed] (18,-1.5) --(18,1.5);
\end{tikzpicture}  
\caption{The orders $\sqsubset $ and $\ll$ on $S_3$ (left) and on $NC_3$ (right).\label{figsqS3}}
\end{figure}
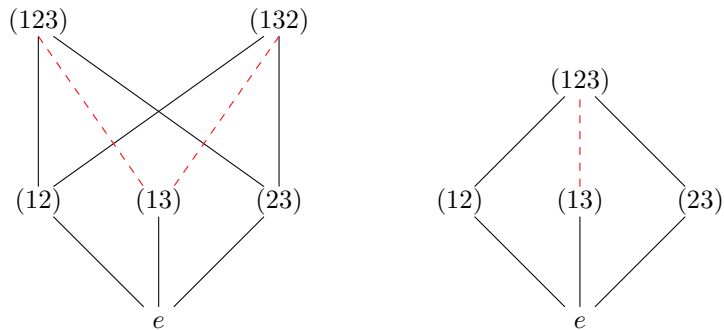

The order $\sqsubset $ and $\ll$ on $S_3$ and on $NC_3$ are shown in Figure~\ref{figsqS3}.
The cover relations of $\sqsubsetdot$ are in black, those of $\lldot$ are in red and dashed.
Note that the underlying undirected graph on $S_3$ is the Cayley graph.  The two possible choices for the Coxeter element give isomorphic posets.

\bigskip

\begin{figure}
\begin{tikzpicture}[scale=.4]
\draw (0,0) node{$(12)$};
\draw (4,4) node{$(13)$};
\draw (8,0) node{$(23)$};
\draw (1,0) --(7,0);
\end{tikzpicture}  
\caption{The graph of the relation $\mathrel{\between_c}$ in $S_3$.\label{figbetween3}}
\end{figure}

\subsection{\texorpdfstring{The case of $S_4$}{The case of S 4}}
Here we have two nonequivalent choices for the Coxeter element: $c=(1234)$ or $c=(1342)$.

\subsubsection{\texorpdfstring{$c=(1234)$}{c=(1234)}}
The poset $NC_4$ with its cover relations drawn as above is in Figure~\ref{figsqS4_1}.  The graph of the relation $\between_c$ is in Figure~\ref{figbetween4}.

\begin{figure}[h!tp]
\begin{tikzpicture}[scale=.5]

\draw (9, 9) node{$(1234)$};
\draw (6, -9) node{$()$};

\draw (0,3) node{$(123)$};
\draw(0,-3) node{$(12)$};

\draw (3,3) node{$(12)(34)$};
\draw (3,-3) node{$(23)$};

\draw (6,3) node{$(234)$};
\draw(6,-3) node{$(34)$};

\draw (9,3) node{$(124)$};
\draw (9,-3) node{$(24)$};

\draw (12,3) node{$(134)$};
\draw(12,-3) node{$(13)$};

\draw (15,3) node{$(14)(23)$};
\draw (15,-3) node{$(14)$};

\draw (0,-2.5) --(0,2.5);
\draw (0,-2.5) --(9,2.5);
\draw (0,-2.5) --(3,2.5);

\draw (3,-2.5) --(0,2.5);
\draw (3,-2.5) --(6,2.5);
\draw (3,-2.5) --(15,2.5);

\draw (6,-2.5) --(3,2.5);
\draw (6,-2.5) --(6,2.5);
\draw (6,-2.5) --(12,2.5);

\draw [red,dashed](9,-2.5) --(6,2.5);
\draw (9,-2.5) --(9,2.5);

\draw (12,-2.5) --(12,2.5);
\draw [red,dashed](12,-2.5) --(0,2.5);

\draw [red,dashed](15,-2.5) --(12,2.5);
\draw [red,dashed](15,-2.5) --(9,2.5);
\draw (15,-2.5) --(15,2.5);

\draw (0,3.5) --(9,8.5);
\draw (3,3.5) --(9,8.5);
\draw (6,3.5) --(9,8.5);
\draw [red,dashed](9,3.5) --(9,8.5);
\draw [red,dashed](12,3.5) --(9,8.5);
\draw [red,dashed](15,3.5) --(9,8.5);
\draw (0,-3.5) --(6,-8.5);
\draw (6,-3.5) --(6,-8.5);
\draw (15,-3.5) --(6,-8.5);
\draw (3,-3.5) --(6,-8.5);
\draw (9,-3.5) --(6,-8.5);
\draw (12,-3.5) --(6,-8.5);
\end{tikzpicture}  
\caption{The order relations $\sqsubset $ and $\ll$ on $NC_4$.\label{figsqS4_1}}
\end{figure}
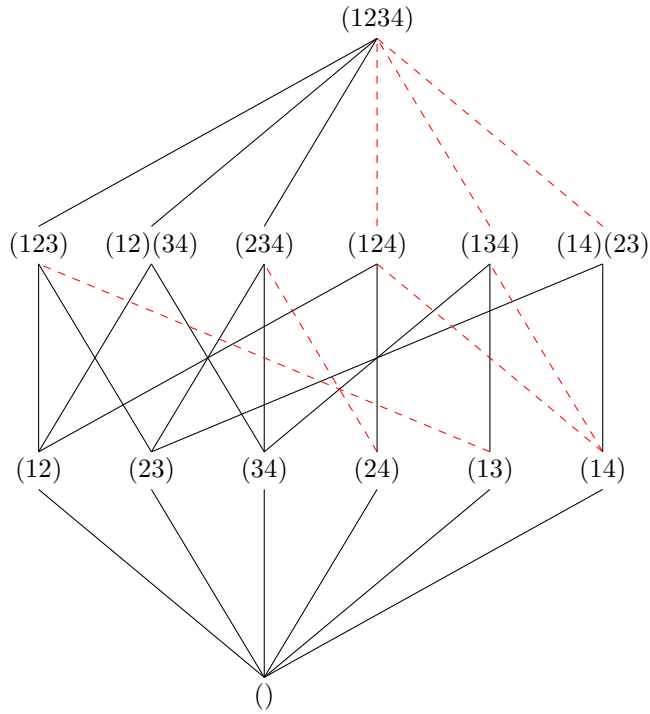

\begin{figure}[h!tp]
\begin{tikzpicture}[scale=.5]
\draw (0,0) node{$(12)$};
\draw (4,0) node{$(34)$};
\draw (2,2) node{$(23)$};
\draw (2,4) node{$(14)$};
\draw (-1.5,-1.5) node{$(24)$};
\draw(5.5,-1.5) node{$(13)$};
\draw (0.5,0.5) --(1.5,1.5);
\draw (2,2.5) --(2,3.5);
\draw (2.5,1.5) --(3.5,.5);
\draw (1,0) --(3,0);
\draw (-1,-1) --(-.5,-.5);
\draw (4.5,-.5) --(5,-1);
\end{tikzpicture}  
\caption{The graph of the relation $\mathrel{\between_c}$ in $S_4$.\label{figbetween4}}
\end{figure}

\subsubsection{\texorpdfstring{$c=(1243)$}{c=(1234)}}
Here we change the Coxeter element for the bipartite element $c=(12)(34)(23)=(1243)$.  The poset $NC_4$ with its cover relations drawn as above is in Figure~\ref{figsqS4_2}.  The graph of the relation $\between_c$ is in Figure~\ref{figbetween4_2}.  We see that the graphs of $\mathrel{\between_c}$ and of $\sqsubset$ for $(1234)$ and $(1243)$ are not isomorphic.  

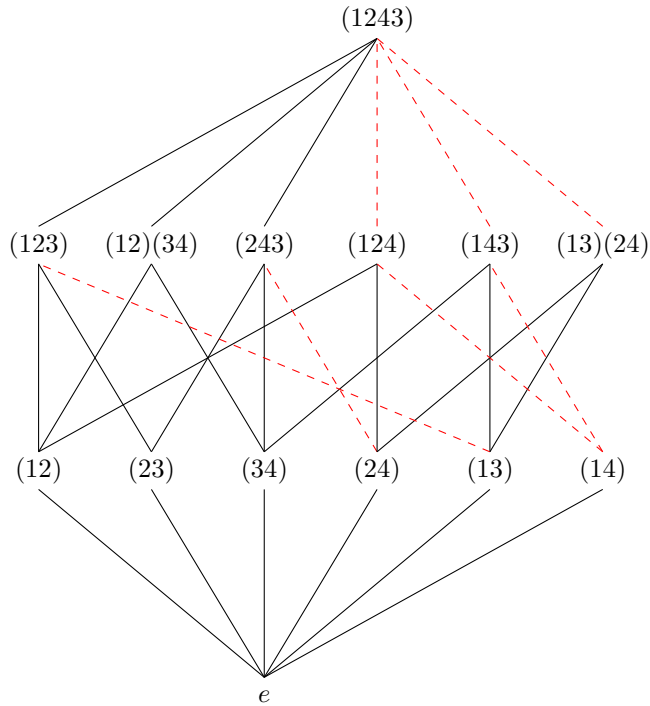
\begin{figure}
\begin{tikzpicture}[scale=.5]

\draw (9, 9) node{$(1243)$};
\draw (6, -9) node{$e$};

\draw (0,3) node{$(123)$};
\draw(0,-3) node{$(12)$};

\draw (3,3) node{$(12)(34)$};
\draw (3,-3) node{$(23)$};

\draw (6,3) node{$(243)$};
\draw(6,-3) node{$(34)$};

\draw (9,3) node{$(124)$};
\draw (9,-3) node{$(24)$};

\draw (12,3) node{$(143)$};
\draw(12,-3) node{$(13)$};

\draw (15,3) node{$(13)(24)$};
\draw (15,-3) node{$(14)$};

\draw (0,-2.5) --(0,2.5);
\draw (0,-2.5) --(9,2.5);
\draw (0,-2.5) --(3,2.5);

\draw (3,-2.5) --(0,2.5);
\draw (3,-2.5) --(6,2.5);

\draw (6,-2.5) --(3,2.5);
\draw (6,-2.5) --(6,2.5);
\draw (6,-2.5) --(12,2.5);

\draw [red,dashed](9,-2.5) --(6,2.5);
\draw (9,-2.5) --(9,2.5);
\draw (9,-2.5) --(15,2.5);

\draw (12,-2.5) --(12,2.5);
\draw [red,dashed](12,-2.5) --(0,2.5);
\draw (12,-2.5) --(15,2.5);

\draw [red,dashed](15,-2.5) --(12,2.5);
\draw [red,dashed](15,-2.5) --(9,2.5);

\draw (0,3.5) --(9,8.5);
\draw (3,3.5) --(9,8.5);
\draw (6,3.5) --(9,8.5);
\draw [red,dashed](9,3.5) --(9,8.5);
\draw [red,dashed](12,3.5) --(9,8.5);
\draw [red,dashed](15,3.5) --(9,8.5);
\draw (0,-3.5) --(6,-8.5);
\draw (6,-3.5) --(6,-8.5);
\draw (15,-3.5) --(6,-8.5);
\draw (3,-3.5) --(6,-8.5);
\draw (9,-3.5) --(6,-8.5);
\draw (12,-3.5) --(6,-8.5);
\end{tikzpicture}  
\caption{The order relations $\sqsubset $ and $\ll$ on $NC(S_4,(1243))$.\label{figsqS4_2}}
\end{figure}

\begin{figure}[h!tp]
\begin{tikzpicture}[scale=.5]

\draw (0,0) node{$(12)$};
\draw (4,0) node{$(34)$};

\draw (2,2) node{$(23)$};
\draw (2,4) node{$(14)$};
\draw (-1.5,-1.5) node{$(24)$};
\draw(5.5,-1.5) node{$(13)$};

\draw (0.5,0.5) --(1.5,1.5);
\draw (2.2,1.5) --(3.5,.5);
\draw (-.5,-1.5) --(4.5,-1.5);
\draw (1,0) --(3,0);
\draw (-1,-1) --(-.5,-.5);
\draw (4.5,-.5) --(5,-1);
\end{tikzpicture}  
\caption{The graph of the relation $\mathrel{\between_c}$ on $S_4$.\label{figbetween4_2}}
\end{figure}

\section{Cluster complex, nonnesting partitions and Chapoton triangles}\label{cluster_sec}

\subsection{The cluster complex}\label{clustersec}

The cluster complexes were introduced by Fomin and Zelevinsky in \cite{fominzelevinsky} as dual to generalized associahedra, in relation with cluster algebras. 
In this paper we will use  the notion of $c$-clusters as defined by Reading \cite[Section~7]{reading}, following Marsh, Reineke and Zelevinsky \cite{marshreinekezelevinsky}.


Let $\Phi_{\geq -1}=(-\Delta)\cup\Pi$ be the set of {\it almost positive roots} (see \cite{fominzelevinsky}).
\begin{prop}[\cite{reading}] \label{cluster_compat}
For each $s \in S$, let $\sigma_s$ denote the bijection from $\Phi_{\geq -1}$ to itself defined by:
\[
  \sigma_s(\alpha) = 
  \begin{cases}
     \alpha & \text{ if } \alpha \in (-\Delta) \backslash \{ -r(s) \}, \\
     s(\alpha) & \text{ if } \alpha\in \Pi \cup \{ -r(s) \}.
  \end{cases}
\]
There exists a unique family of symmetric binary relations on $\Phi_{\geq -1}$ indexed by standard Coxeter elements, denoted by $\mathrel{\|_c}$, such that:
\begin{itemize}
 \item if $\alpha,\beta \in -\Delta$, then $\alpha \mathrel{\|_c} \beta$,
 \item if $\alpha\in -\Delta$ and $\beta \in \Pi$, then $\alpha \mathrel{\|_c} \beta$ if and only if $r(-\alpha) \notin \supp(r(\beta))$,
 \item if $s\in S$ is a left descent of $c$, then $\alpha \mathrel{\|_c} \beta \Longleftrightarrow \sigma_s(\alpha) \mathrel{\|_{scs}} \sigma_s(\beta)$.
\end{itemize}
\end{prop}

\begin{rema} Recall Remark \ref{onlycommutations} i) on the use of the symbol $ \mathrel{\|_c}$.
\end{rema}
The definition of $\mathrel{\|_c}$ is not fully explicit and uniqueness relies on Lemma~\ref{scs_moves}.
A more direct characterization of $\|_c$ is given in \cite{ceballoslabbestump} using the subword complex.  Another will be given in Proposition~\ref{clust_pos}.  Let us first review some results from \cite{reading,readingspeyer}.

\begin{defi}[\cite{reading}]\label{defclusterrel}
The $c$-cluster complex $\Upsilon(W,c)$ is the flag simplicial complex on the vertex set $\Phi_{\geq -1}$ 
defined by the relation $\mathrel{\|_c}$. A {\it $c$-cluster} is a maximal face  of $\Upsilon(W,c)$.
\end{defi}

This complex is pure, so that every $c$-cluster has dimension $n$.  Clearly, the map $\sigma_s$ provides an isomorphism between $\Upsilon(W,c)$ and $\Upsilon(W,scs)$ if $s\in S$ is a left descent of $c$.  So the isomorphism type of $\Upsilon(W,c)$ does not depend on $c$, see \cite{reading,readingspeyer} for details.  

The face generating function with respect to cardinality is:
\begin{align} \label{narayana_faces}
   \sum_{F\in \Upsilon(W,c) } x^{\# F} = \sum_{k=0}^n \Nar_k(W) x^{n-k}(1+x)^k.
\end{align}
In particular, taking the coefficient of $x^n$ shows that the number of clusters is $\Cat(W)$.
Moreover \eqref{narayana_faces} implies that the integers $\Nar_k(W)$ are the entries of the $h$-vector of $\Upsilon(W,c)$.

\begin{defi}
 The {\it positive part} of  $\Upsilon(W,c)$, denoted by $\Upsilon^+(W,c)$, is the flag simplicial complex with vertex set $\Pi$ and 
 compatibility relation $\mathrel{\|_c}$.
\end{defi}

Clearly, $\Upsilon^+(W,c)$ is a (full) subcomplex of $\Upsilon(W,c)$. Note that it may happen that $\Upsilon^+(W,c_1)$ and $\Upsilon^+(W,c_2)$ are not isomorphic if $c_1,c_2$ are two different standard Coxeter elements. The face generating function with respect to cardinality is:
\begin{align} \label{narayana_faces_pos}
   \sum_{F\in \Upsilon^+(W,c) } x^{\# F} = \sum_{k=0}^n \Nar^+_k(W) x^{n-k}(1+x)^k.
\end{align}
In particular, taking the coefficient of $x^n$ shows that the number of positive $c$-clusters is $\Cat^+(W)$ and 
 the integers $\Nar^+_k(W)$ are the entries of the $h$-vector of $\Upsilon^+(W,c)$.

We give a direct characterization of $\|_c$ in Proposition~\ref{clust_pos} below. It is an extension of Brady and Watt's results in the bipartite case \cite{bradywatt2} to any standard Coxeter element.  We begin with a lemma.

\begin{lemm}\label{char_cluster}
Suppose $s\in S$ is a left descent of $c$, and let $t\in T\backslash\{s\}$. Then the following two conditions are equivalent:
\begin{itemize}
 \item $\langle r(s)|r(t)\rangle \geq 0$, and $s,t$ are $c$-noncrossing (see Definition \ref{noncrois}),
 \item $s\notin \supp(sts)$.
\end{itemize}
\end{lemm}

\begin{proof}
We begin by showing that the first condition implies the second one.

\noindent {\bf $\Longrightarrow$, case 1: } $st\leq c$.  By Proposition~\ref{leftdesc}, we have $t\leq sc$ and $t\in  W_{\langle s \rangle}$ therefore $s\notin \supp(t)$ and $\langle r(s) | r(t) \rangle \leq 0$.  Since we already have $\langle r(s) | r(t) \rangle \geq 0$, we get $\langle r(s) | r(t) \rangle = 0$, which means $st=ts$. 
Then $s\notin \supp(sts)$ follows from $s\notin \supp(t)$ since $sts=t$.

\noindent {\bf $\Longrightarrow$, case 2: } $ts\leq c$.  We have $\ell_T(ts)=2$ since $s\neq t$, therefore $s\leq ts$. By Proposition~\ref{leftdesc}, we get $sts \leq sc\in W_{\langle s\rangle}$.  It follows that $s\notin \supp(sts)$.

\noindent {\bf $\Longleftarrow$, case 1: } $st=ts$.  First note that we have $\langle r(s) | r(t) \rangle = 0$ in this case, so in particular it is $\geq 0$. Then, we have $sts=t$ so $s\notin \supp(t)$. It follows that $t \leq sc$ and $st\leq c$ by Proposition~\ref{leftdesc}.

\noindent {\bf $\Longleftarrow$, case 2: } $st \neq ts$.  We have $sts\in W_{\langle s \rangle} $, so $ts=s(sts) \leq c$ by Proposition~\ref{leftdesc}. Also, from $sts\in W_{\langle s \rangle } $ we get $\langle r(s) | r(sts) \rangle \leq 0$.  Since $s$ is in the orthogonal group, it preserves the scalar product and we get $\langle s (r(s)) | s (r(sts)) \rangle\leq 0$, hence  $\langle -r(s) | r(t) \rangle \leq 0$ and $\langle r(s) | r(t) \rangle \geq 0$ as needed.
\end{proof}

Using the above lemma we can now give a more explicit characterization of the relation $\mathrel{\|_c}$.

\begin{prop} \label{clust_pos}
Let $\alpha,\beta \in \Pi$, then  $\alpha \mathrel{\|_c} \beta$ if and only if the following two conditions hold:
\begin{itemize}
 \item  $r^{-1}(\alpha), \,r^{-1}(\beta)$ are $c$-noncrossing,
 \item $\langle \alpha|\beta\rangle \geq 0$.
\end{itemize}

\end{prop}

\begin{proof}
 Let us introduce a family of binary relations $\mathrel{\hat\|_c}$ on $\Phi_{\geq -1}$ by:
 \begin{itemize}
  \item if $\alpha,\beta\in\Pi$ one has $\alpha \mathrel{\hat\|_c} \beta$ if and only if 
   $\langle \alpha | \beta \rangle \geq 0$ and $r^{-1}(\alpha), \,r^{-1}(\beta)$ are $c$-noncrossing,
  \item otherwise, $\alpha \mathrel{\hat\|_c} \beta$ if and only if $  \alpha \mathrel{\|_c} \beta$.
 \end{itemize}
We  show that the relations $\mathrel{\hat\|_c}$ satisfy the rules of $\mathrel{\|_c}$ given in Definition~\ref{cluster_compat}. By  uniqueness it  follows that
$\mathrel{\hat\|_c} \; = \; \mathrel{\|_c}$.
The first two points in Definition~\ref{cluster_compat} are obviously satisfied by the family $\mathrel{\hat\|_c}$ therefore we just need 
 to show that they satisfy the third point. Let $s$ be a left descent of $c$ and $\alpha,\beta\in \Phi_{\geq -1}$.

\noindent {\bf Case 1: } $\alpha,\beta\in (-\Delta)\backslash\{-r(s)\} $.
We have then $\sigma_s(\alpha)=\alpha$, $\sigma_s(\beta)=\beta$, it follows that $\alpha \mathrel{\hat\|_c} \beta$ and $\sigma_s(\alpha) \mathrel{\hat\|_{scs}} \sigma_s(\beta)$
both hold.

\noindent {\bf Case 2: } $\alpha\in (-\Delta)\backslash\{r(s)\},\ \beta =-r(s)$. One has $\alpha \mathrel{\hat\|_c} \beta$, $\sigma_s(\alpha)=\alpha$, $\sigma_s(\beta)=-\beta>0$ and $-\alpha\notin\supp (-\beta)=\{s\}$ therefore
$\sigma_s(\alpha)\mathrel{\hat\|_{scs}}\sigma_s(\beta)$.

\noindent {\bf Case 3: } $\alpha,\beta\in \Pi\backslash\{r(s)\} $.
In this case, one has  $\sigma_s(\alpha) = s(\alpha)$ and $\sigma_s(\beta) = s( \beta)$. Since $s\in O(V)$ preserves the scalar product, $\langle \alpha| \beta \rangle =\langle \sigma_s(\alpha) | \sigma_s(\beta) \rangle $, moreover $r^{-1}(\sigma_s(\alpha)) = s  r^{-1}(\alpha) s $ and $r^{-1}(\sigma_s(\beta)) = s  r^{-1}(\beta) s $. Therefore we have
\[
  r^{-1}(\sigma_s(\alpha)) r^{-1}(\sigma_s(\beta)) = s r^{-1}(\alpha)  r^{-1}(\beta) s,
\] 
and the conclusion follows since $\leq$ is invariant under conjugation.
 
\noindent {\bf Case 4: } $\alpha=-r(s),\beta\in \Pi\backslash\{r(s)\} $. This case follows from Lemma~\ref{char_cluster} putting $s=r^{-1}(\alpha)$ and $t=sr^{-1}(\beta)s$.

\noindent {\bf Case 5: } $\alpha=r(s),\beta\in \Pi\backslash\{r(s)\} $. Again use  Lemma~\ref{char_cluster} for $s=r^{-1}(\alpha)$ and $t=r^{-1}(\beta)$.

Since the relations $\mathrel{\|_c}$ and $\mathrel{\hat\|_c}$ are symmetric we have covered all cases.
\end{proof}

By Proposition~\ref{reading1prodsimples}, if $w\leq c$ then $w$ is a standard Coxeter element of $\Gamma(w)$ so that
the complex $\Upsilon(\Gamma(w),w)$ is well defined.  An interesting consequence of Proposition~\ref{clust_pos} is that 
for $t_1,t_2\in \Pi(\Gamma(w))$, we have $\{t_1,t_2\}\in \Upsilon(\Gamma(w),w)$ if and only if $\{t_1,t_2\}\in \Upsilon(W,c)$.
This will be used in Proposition~\ref{face_to_cluster} to show that any positive face in $\Upsilon(W,c)$ is a cluster in 
$\Upsilon(\Gamma(w),w)$ for some $w\leq c$, a result which is easily extended to all faces.

Note  the similarity between Proposition~\ref{clust_pos} and Definition~\ref{defi_Xi}
therefore between the two flag simplicial complexes $\Upsilon^+(W,c)$ and $\Xi(W,c)$.
Their vertex sets are respectively $T$ and $\Pi$, and are in bijection via the map $r$, 
besides this the only difference is the required sign of the scalar product. Despite this they have very different properties:
unlike $\Upsilon^+(W,c)$, the simplicial complex $\Xi(W,c)$ is not pure and its topology can be rather complicated.

Eventually, the following result will be useful:

\begin{prop}\label{indexingface}
If $F \in \Upsilon^+(W,c)$, there is an indexing $F=\{ \alpha_1 , \dots , \alpha_k \}$ such that $r(\alpha_1)\cdots r(\alpha_k) \leq c$.
\end{prop}

\begin{proof}
 It is sufficient to check this statement for $k=n$ and $r(\alpha_1) \cdots r(\alpha_n) = c$.  This case follows from \cite[Proposition~2.8]{ceballoslabbestump}. 
 
 Alternatively, this follows from the work of Brady and Watt~\cite{bradywatt2}. In the case where $c$ is a bipartite Coxeter element, \cite[Note~3.3]{bradywatt2} implies the result. It remains to show that this is preserved under the moves $c\to scs$ as in Lemma~\ref{scs_moves}. We omit details.
\end{proof}

\subsection{\texorpdfstring{The relation $ \mathrel{\|_c}$ and the orders $\sqsubset$ and $\ll$}{The relation || and the orders [ and <<}}
In the simply laced case (i.e.~types $A_n$, $D_n$ and $E_n$ in the classification), using Proposition~\ref{clust_pos}  the compatibility relation $\|_c$ on $\Pi$ can be completely rephrased in terms of the orders $\sqsubset$ and $\ll$, without
using roots or even the group structure. This relies on three observations:
\begin{itemize}
 \item Each rank 2 parabolic subgroup has type $A_2$ or $A_1^2$ in the classification. For any pair of distinct reflections in such a subgroup,
       we can take the product in some order to get a given Coxeter element. It means that the condition $tu\leq c$ or $ut\leq c$ for $t,u\in T$ is
       equivalent to the existence of $v\leq c$ such that $\rk(v)=2$, and $t\leq v$, $u\leq v$.
 \item The condition $\langle r(t) | r(u) \rangle > 0 $ means that $\{t,u\}$ is not the simple system of $\Gamma(v)$, so it is equivalent
       to the fact that either $t\ll v$ or $u\ll v$ (with $v$ as above).
 \item The condition $\langle r(t) | r(u) \rangle = 0 $ means $tu=ut$, it is equivalent to the fact that the interval $[e,tu]_\leq$
       contains only $e$, $t$, $u$, $tu$. 
\end{itemize}
It follows that $\alpha \mathrel{\|_c} \beta $ is equivalent to:
\begin{itemize}
 \item $r^{-1}(\alpha)$ and $r^{-1}(\beta)$ have a least upper bound $v$ of rank $2$,
 \item Either $[e,v]_{\leq} = \{e,r^{-1}(\alpha),r^{-1}(\beta),v\}$, or $r^{-1}(\alpha)\ll v$ or $r^{-1}(\beta)\ll v$.
\end{itemize}

\subsection{Nonnesting partitions}

When $W$ is a Weyl group, we have an associated root poset, defined in terms of the {\it crystallographic root system} of $W$.  We defined roots to be unit vectors, which is not convenient for crystallographic root system, so we only sketch the definition here and refer to \cite[Chapter~2.9]{humphreys}.
The idea is to allow to have roots of different lengths, so we consider positive numbers $(a_\alpha)_{\alpha\in\Pi}$ and assume that the vectors $a_\alpha \alpha$ satisfies, for all $\alpha,\beta\in\Pi$:
\[
   \frac{2 \langle a_\alpha \alpha | a_\beta \beta \rangle}{\langle a_\beta \beta | a_\beta \beta \rangle} \in \mathbb{Z}.
\]

\begin{defi}
 The partial order $\preccurlyeq$ on $\Pi$ is defined by $\alpha \preccurlyeq \beta$ if $a_\beta \beta- a_\alpha\alpha$ is a linear combination of simple roots with nonnegative coefficients and $(\Pi,\preccurlyeq)$ is called the {\it root poset} of $W$.  We denote $ NN(W)$ the set of antichains of $(\Pi,\preccurlyeq)$, and such an antichain is called a {\it nonnesting partition}.
\end{defi}

Note that in the case where $W$ is not a Weyl group (i.e.~one of $I_2(m)$, $H_3$, and $H_4$ in the classification),
there are {\it ad hoc} definitions of posets having some of the expected properties of a root poset.  However this fully works only for $I_2(m)$ and $H_3$.
We don't discuss that in more details, see \cite[Section~5.4.1]{armstrong} and \cite{cuntzstump}.

Just as $ NC(W,c)$ and $\Upsilon(W,c)$, the set $ NN(W)$ is a flag simplicial complex.
Its vertex set is $\Pi$, and two vertices are compatible if they are not comparable in the root poset.
This complex is not pure.

\begin{defi}
The {\it support} of a nonnesting partition $A$ is:
\[
  \supp(A) = \{ \delta \in \Delta \, : \,  \exists \alpha \in A, \; \delta \preccurlyeq \alpha \}.
\]
\end{defi}

It is similar to the notion of support a group element $w\in W$, so there should be no confusion.
For example, note that for $A\in NN(W)$, we have $s\notin\supp(A)$ if and only if $\pi\in NN(W_{\langle s \rangle})$.

\subsection{Enumeration of full reflections}
\label{sec_fullrefs}

The number of $t\in T$ with full support is $\Nar^+_1(W)$.  Chapoton \cite{chapoton2} obtained the formula
\begin{equation}  \label{nb_reffull}
  \Nar_1^+(W) = \frac{nh}{|W|}\prod_{i=2}^n (e_i-1)
\end{equation}
by a case by case verification, and he also conjectured a representation theoretical interpretation.  We show that this formula can be obtained from the properties of $\ll$. 

By Proposition~\ref{llfull}, a reflection $t \in T$ is full if and only if $t\ll c$, which is equivalent to the existence of a maximal chain
\begin{align} \label{maxchains_reffull}
   e \sqsubsetdot w_1 \lldot w_2 \lldot \dots \lldot w_n = c
\end{align}
with $t=w_1$ in $\ NC(W,c)$.  Note that these maximal chains are those with the maximal number of $\lldot$, as we have $e\sqsubsetdot t$ for $t\in T$.
So their number is the dominant coefficient in \eqref{enum_maxchains}. Noticing that $d_n+q(h-d_n)=h$ is a constant, this dominant coefficient is:
\begin{align} \label{maxchains_reffull2}
  \frac{n!\times h}{|W|} \prod_{i=1}^{n-1} ( h-d_i ).
\end{align}
The number of maximal chains as in \eqref{maxchains_reffull} is also $(n-1)!$ times the number of full reflections: for a given $w_1$ the possible choices for $w_2,\dots,w_n$ are the maximal chains in $[w_1,c]_\ll$ which is a boolean lattice of rank $n-1$ by Corollary~\ref{booleanideals} and their number is $(n-1)!$.  To get the right hand side of \eqref{nb_reffull} from \eqref{maxchains_reffull2} divided by $(n-1)!$, it remains only to use the equalities $d_i=e_i+1$, $h-e_i=e_{n+1-i}$.

 Fomin and Reading in \cite[Section~13.4]{fominreading} asked for  a better combinatorial way to relate full reflections with objects counted by Fuß-Catalan numbers ($k$-noncrossing partitions, generalized clusters, see~\cite{armstrong}).  Our derivation of the formula does not give a full answer but recasts the problem in a more general form: have a better combinatorial way to relate Fuß-Catalan objects and the generating function $M(w,q)$, thus explaining the relation~\eqref{maxchain_fusscat}.

\subsection{Chapoton triangles}

\label{triangles}

The $F$-, $M$- and $H$-triangles are polynomials in two variables defined respectively in terms of $\Upsilon(W,c)$, 
the Möbius function of $ NC(W,c)$, and $ NN(W)$. The $F=M$ and $H=M$ theorems state that these polynomials are related 
to each other by invertible substitutions (there is an $F=H$ theorem as an immediate consequence of the other two). 
They were conjectured by Chapoton in \cite{chapoton2} and \cite{chapoton1}, respectively, 
and proved by Athanasiadis~\cite{athanasiadis} and Thiel~\cite{thiel}, respectively. 
See also \cite{armstrong} for generalizations.

\begin{defi}
Let $\mu$ denote the Möbius function of $ NC(W,c)$. The $M$-triangle is the polynomial:
\begin{equation*}
  M(x,y) = \sum_{\substack{ \alpha,\beta \in  NC(W,c)  \\  \alpha \leq \beta }} \mu(\alpha,\beta) x^{\rk(\alpha)} (-y)^{\rk(\beta)-\rk(\alpha)}.
\end{equation*}
The $F$-triangle is the polynomial:
\begin{equation*}
  F(x,y) = \sum_{F\in \Upsilon(W,c)} x^{\# ( F \cap ( -\Delta ) )} y^{ \#(F \cap \Pi)}.
\end{equation*}
The $H$-triangle is the polynomial:
\begin{equation*}
  H(x,y) = \sum_{ A \in  NN(W) } x^{\#(A\cap \Delta)} y^{\#(A \cap (\Pi\backslash\Delta))}.
\end{equation*}
\end{defi}

There exist slightly different conventions in the literature for these polynomials. Here and also for related polynomials in the sequel, 
we always take a convention ensuring that they have nonnegative integer coefficients, and total degree $n$.

The $F=M$ and $H=M$ theorems relates the three polynomials as follows:
\begin{align} \label{eqFHM}
             F(x,y) &= (1+y)^n H\Big( \tfrac{x}{1+y} , \tfrac{y}{1+y} \Big) =(1+x)^n M\Big( \tfrac{x}{1+x} , \tfrac{y-x}{1+x} \Big), \\
             H(x,y) &= (1-y)^n F \Big( \tfrac{x}{1-y} , \tfrac{y}{1-y} \Big) = (1+x-y)^n M\Big( \tfrac{x}{1+x-y}, \tfrac{y-x}{1+x-y} \Big), \\
\label{MFH}  M(x,y) &= (1+y)^n H \Big( \tfrac{x}{1+y} , \tfrac{y+x}{1+y} \Big) = (1-x)^n F \Big( \tfrac{x}{1-x} , \tfrac{y+x}{1-x} \Big).
\end{align}
See \cite{athanasiadis,chapoton1,chapoton2,thiel} for details.  Note that all these relations suggest  considering 
homogeneous polynomials in three variables, rather than only two variables.  Then the relations become mere shifts of the variables.
This idea will be used in Section~\ref{sec_genFHM}.
%

Also, the polynomials satisfy a symmetry property.  The self-duality of $ NC(W,c)$ gives
\begin{equation}
   M(x,y) = x^n M(\tfrac 1x , \tfrac 1y ).
\end{equation}
and, using \eqref{eqFHM},
\begin{equation} \label{symm_FH}
   F(x,y) = (-1)^n F(-1-x,-1-y), \quad H(x-1,y) = y^n H(\tfrac yx -1 , \tfrac 1y ).
\end{equation}

From the definition, it can be seen that these polynomials contains Catalan, positive Catalan, Narayana, and positive Narayana numbers as special cases.
Let us also mention another interesting specialization: $H(-1,1)$ is the {\it double-positive Catalan number} introduced in \cite{barnardreading},
more generally the coefficients of $H(-1,y)$ are the {\it double-positive Narayana numbers}.  See \cite[Proposition~4.6]{barnardreading} for the interpretation in terms of nonnesting partitions. Using \eqref{eqFHM}, we have
\[
  H(-1,y) = (1-y)^n F\big( \tfrac{-1}{1-y} , \tfrac{1}{1-y} \big)
\]
and the latter expression can be interpreted as the {\it local $h$-polynomial} of $\Upsilon^+(W,c)$, see \cite[Remark~4.7]{barnardreading} and Athanasiadis and Savvidou~\cite{athanasiadissavvidou}.

\section{\texorpdfstring{Intervals for $\sqsubset$ and $\ll$ and the cluster complex}{Intervals for [ and << and the cluster complex}}
\label{interval_cluster_sec}

\subsection{\texorpdfstring{Intervals for $\sqsubset$}{Intervals for [}}
\subsubsection{Counting intervals}
If $v,w\in NC(W,c)$ and $v\sqsubset w$, we denote by $[v,w]_{\sqsubset}$ the interval: 
\[
  [v,w]_{\sqsubset} = \big\{ x \in NC(W,c) \; : \; v\sqsubset x\sqsubset w   \big\}.
\]
Define the {\sl height} of an interval $[v,w]_{\sqsubset}$ as $rk(w)-rk(v)$.

The identity $e$ is the smallest element in $NC(W,c)$ for $\sqsubset$. 
Moreover, if $w$ has rank $r$ we have seen that the interval $[e,w]_\sqsubset$ is a boolean lattice with $2^r$ elements. 
Therefore, if $r\geq k$  the number of elements $v\sqsubset w$ with $\rk(w)-\rk(v)=k$ is equal to $\binom{r}{k}$. 
Since the number of elements of $NC(W,c)$ of rank $r$ is $\Nar_r(W)$, we have the following (here $n$ is the rank of $W$).

\begin{prop} 
The number of intervals of height $k$  for $\sqsubset$ is equal to
\begin{equation}\label{intsqsub}
  \sum_{r=k}^n\Nar(W,r)\binom{r}{k}.
\end{equation}
\end{prop}

In the case of the symmetric group $\mathfrak{S}_n$, the total number of intervals for $\sqsubset$ is the {\it small Schröder number} $s_n$, defined by
\[
  \sum_{n\geq 0} s_n z^n  =  \frac{1+x-\sqrt{1-6x+x^2}}{4x}.
\]

\subsubsection{Relation with the cluster complex and the associahedron}
We now relate $( NC(W,c),\sqsubset)$ with the cluster complex $\Upsilon(W,c)$ and the associahedron, a polytope whose face complex is dual to $\Upsilon(W,c)$~\cite{fominzelevinsky}.  It is known that the $h$-vector of the associahedron is equal to the sequence of $W$-Narayana numbers, see~\cite{athanasiadisbradymaccammondwatt} (or \eqref{narayana_faces}, using the duality). The next statement follows readily from \eqref{intsqsub} and the relation between the $f$- and  $h$-vectors.

\begin{prop} \label{clust_sqsub}
The number of intervals of height $k$ for $\sqsubset$ is equal to the  number of faces of dimension $k$ of the $W$-associahedron, or to the number of faces of cardinality $n-k$ of the cluster complex.
\end{prop}

For  particular values of $k$ there are bijective proofs of the equality in Proposition~\ref{clust_sqsub}.

\begin{itemize}
 \item For $k=n$: there is a unique interval of height $n$ which is $ INT(W,c)$ and a unique face of dimension $n$ in the associahedron.

 \item For $k=0$: the number of vertices of the associahedron is equal to number of non-crossing partitions, or $W$-Catalan number, which is also the number of intervals of height $0$ for $\sqsubset$. Bijective proofs of this fact have appeared in the litterature, see Section \ref{secbij}.

 \item For $k=n-1$: the  number of intervals of height $n-1$ for $\sqsubset$ is equal to the number of almost positive roots, which is the number of vertices of the cluster complex. Indeed this is easy to check directly since these intervals fall into two categories: 

 \begin{itemize}
 \item the intervals $[e,\pi]$ where $\pi\in NC(W,c)$ has rank $n-1$. The Kreweras complement gives a bijection between this  set and the elements of rank $1$ in $NC(W,c)$ which are the reflections, in bijection with the positive roots.
 \item the intervals 
 $[s,c]_\sqsubset$ where $s\in S$ which are in bijection with the negative simple roots.
 \end{itemize}
\end{itemize}
Observe that the study of the order $\sqsubset$ leads naturally to the notion of almost positive root, which is fundamental to the theory of clusters.

It would be interesting to give a bijective proof of Proposition \ref{clust_sqsub} for other values of $k$. In the case of a bipartite Coxeter element, we will give,  in Section~\ref{bijsqsub}, a bijection between faces of size $n-k$ of the cluster complex and intervals of height $k$ for $\sqsubset $.

\subsection{\texorpdfstring{Intervals for $\ll$}{Intervals for <<}}

In the case of the symmetric group $\mathfrak{S}_n$ with the standard Coxeter element $c=(1,2,\dots,n)$, the number of intervals $v\ll w$ was computed in \cite{nica} and \cite{petrullo}, it is the so-called {\it large Schröder number} $S_{n-1}$ defined by
\[
  \sum_{n\geq 0} S_n z^n  =  \frac{1-x-\sqrt{1-6x+x^2}}{2x}.
\]
They are related to small Schröder numbers by $S_n = 2s_n$ if $n\geq 1$.

In the general case we will again show a connection with the cluster complex.  Refining the enumeration of intervals, we consider a two variable polynomial:
\begin{equation} \label{defI}
   I(x,y) = \sum_{\substack{ v,w \in  NC(W,c)  \\  v \ll w }} x^{\rk(v)} y^{\rk(w)-\rk(v)}.
\end{equation}
It turns out to be related with the polynomial $M(x,y)$ as follows.

\begin{theo} \label{theoIeqM}
We have:
\begin{align} \label{IeqM}
    I(x,y) = M(x-y,y), \qquad M(x,y) = I(x+y,y).
\end{align}
\end{theo}

To prove that, first consider, for $w\in NC(W,c)$:
\begin{equation} \label{Ibeta}
   I_{w}(x) = \sum_{\substack{ v \in  NC \\ v \ll w }} x^{\rk(v)}
\end{equation}
and 
\begin{equation} \label{Mbeta}
   M_{w}(x) = (-1)^{\rk(w)} \sum_{\substack{ v \in  NC \\ v \leq w}} \mu(v,w) (-x)^{\rk(v)} .
\end{equation}

\begin{prop} \label{refinementIM}
We have $ I_{w}(x+1) =  M_{w}(x) $.
\end{prop}

\begin{proof}
The case $w=c$ is sufficient, because we can see $w$ itself as a standard Coxeter element of $\Gamma(w)$
and prove the identity in this subgroup. Then $I_c(x)$ is the rank generating function of $ NCF(W,c)$, it is therefore given by the positive Narayana numbers and it remains to prove:
\[
  \sum_{k=1}^n \Nar^+_k (1+x)^k
  =
  (-1)^{n} \sum_{ v \in  NC(W,c) } \mu(v,c) (-x)^{\rk(v)}.
\]
The right hand side is, using Kreweras complement:
\[
     \sum_{ v \in  NC(W,c) } \mu(e, K(v) ) (-x)^{\rk(v)} 
  =  \sum_{ v \in  NC(W,c) } \mu(e, v ) (-x)^{n-\rk(v)} 
  = x^n M(0,\tfrac 1x).
\]
This is also $x^n F(0,\frac 1x)$, via \eqref{eqFHM}.  The $h$-vector of $\Upsilon^+(W,c)$ is given by positive Narayana numbers, see \eqref{narayana_faces_pos}, and we get the equality.
\end{proof}


\begin{proof}[Proof of Theorem~\ref{theoIeqM}]
Using Proposition~\ref{refinementIM}, we have:
\[
  I(x,y) = \sum_{w\in NC(W,c)} I_w(\tfrac xy) y^{\rk(w)} = \sum_{w\in NC(W,c)} M_w(\tfrac xy -1) y^{\rk(w)} = M(x-y,y).
\]
\end{proof}

Combining Equations~\eqref{MFH} and~\eqref{IeqM}, we obtain the relations between $I(x,y)$ and $H(x,y)$:
\begin{align*}
  I(x,y) &= (1+y)^n H \left( \frac{x-y}{1+y} , \frac{x}{1+y} \right), \\
  H(x,y) &= (1+x-y)^n I\left( \frac{y}{1+x-y} , \frac{y-x}{1+x-y} \right),
\end{align*}
and the relations between $I(x,y)$ and $F(x,y)$:
\begin{align*}
  I(x,y) &= (1-x+y)^n F \left( \frac{x-y}{1-x+y} , \frac{x}{1-x+y} \right),  \\
  F(x,y) &= (1+x)^n I \left( \frac{y}{1+x} , \frac{y-x}{1+x} \right).
\end{align*}
In particular, $I(x,x)=F(0,x)$ gives the following:

\begin{coro} \label{enum_intervals_faces}
We have
\[
  \sum_{\substack{ v,w \in  NC(W,c) \\ v \ll w }} x^{\rk(w)} = \sum_{ F \in \Upsilon^+(W,c) } x^{\# F}.
\]
In particular, the number of intervals $v\ll w$ in $ NC(W,c)$ is $\# \Upsilon^+(W,c)$.
\end{coro}

We will give a bijective proof of the latter fact in the next section.

\section{\texorpdfstring{The bijection between positive faces of the cluster complex and intervals}{The bijection between positive faces of the cluster complex and intervals}}
\label{secbij}

In order to  give a bijection between $\Upsilon^+(W,c)$ and intervals $v\ll w$ in $ NC(W,c)$, an important ingredient
is a bijection between clusters and noncrossing partitions. Such bijections have been described by Reading \cite{reading}
and Athanasiadis et al.~\cite{athanasiadisbradymaccammondwatt}. Although the one from \cite{athanasiadisbradymaccammondwatt}
was only stated in the case of a bipartite Coxeter element, it turns out to be particularly adapted to the present situation
as it can be rephrased in terms of the orders $\sqsubset$ and $\ll$. In fact, using properties of $\ll$ and $\sqsubset$
we will  extend the   bijection  to cover the case of all standard Coxeter elements.

\subsection{The bijection between clusters and noncrossing partitions.}
We give here a bijection $\Psi_{W,c}$ between $NCF(W,c)$ and clusters in $\Upsilon^+(W,c)$.
Note that it is straightforward to extend such  a bijection  to a bijection from $ NC(W,c)$ to clusters in $\Upsilon(W,c)$:
if  $w\in NC(W,c)$ with $J = \supp(w)$, 
then the image of $w$ is
\[
   \Psi_{W_J, \overline w }(w) \cup \{ -r(s) \, :\, s\notin J \}.
\]
Recall that $\overline w$ is the interval partition such that $\Gamma(\overline w)=W_J$ and it is indeed a standard Coxeter element of $W_J$.
In the case of the bipartite Coxeter element, the bijection that we obtain  is actually the inverse of the bijection 
in \cite[Sections~5-6]{athanasiadisbradymaccammondwatt}.

\begin{defi} Let $w\in  NC(W,c)$ then we define
\begin{align*}
   \Phi_{W,c} (w) &=  \big( \Inv_R(w)\cap \Gamma(w) \big) \cup \big( \Inv_L(K(w))\cap \Gamma(K(w)) \big)\subset T, \\
   \Psi_{W,c} (w) &=  r(\Phi_{W,c} (w))\subset \Pi.
\end{align*}
\end{defi}
Let us  write $w$ and $K(w)$ as products of their associated simple generators as in Proposition~\ref{reading1prodsimples}:
\[
  w = {\mathfrak s}_1 \cdots {\mathfrak s}_k, \qquad K(w) = {\mathfrak s}_{k+1}\cdots {\mathfrak s}_{n},
\]
then one has  
\[ 
  c={\mathfrak s}_1\cdots {\mathfrak s}_n, 
\]
and, using \eqref{leftinv},
 \[
  \Phi_{W,c}(w)= \{ t_1,\dots,t_n \}
\]
where:
\[
  t_i = 
  \begin{cases}
   {(\mathfrak s}_k  {\mathfrak s}_{k-1} \cdots{\mathfrak s}_{k+2-i}) {\mathfrak s}_{k+1-i} ({\mathfrak s}_k  {\mathfrak s}_{k-1} \cdots{\mathfrak s}_{k+2-i})^{-1} & \text{ if } 1\leq i\leq k, \\
   ({\mathfrak s}_{k+1}  {\mathfrak s}_{k+2}\cdots {\mathfrak s}_{i-1}){\mathfrak s}_{i} ({\mathfrak s}_{k+1}  {\mathfrak s}_{k+2}\cdots {\mathfrak s}_{i-1})^{-1}& \text{ if } k+1\leq i \leq n.
  \end{cases}
\]
It follows that  
\[
\Psi_{W,c}(w)= \{r( t_1),\dots,r(t_n) \}\subset \Pi.
\] 
\begin{theo} \label{theo_bijcluster}
The map $\Psi_{W,c}$ is a bijection from $ NCF(W,c)$ to the set of $c$-clusters in $\Upsilon^+(W,c)$.
\end{theo}

First observe that, by a straightforward computation:
\begin{align*}
t_1\cdots t_k &=w,  \\
t_{k+1} \cdots t_n &= K(w), \\
t_1\cdots t_n &= c.
\end{align*}

\begin{lemm} \label{lemma_relationsdot}
Let $w\in  NCF(W,c)$ of rank $k$, $t_i$ as above, and 
 let $1\leq i\leq k$ and $k+1\leq j \leq n$. Then we have $w t_i \sqsubsetdot w$, $ t_j K(w) \sqsubsetdot K(w) $, and $w\lldot w t_j$.
\end{lemm}

\begin{proof}
 An easy computation gives $w t_i = {\mathfrak s}_1 \cdots {\mathfrak s}_{k-i} {\mathfrak s}_{k+2-i} \cdots {\mathfrak s}_k$, whence the first inequality.
 The second is obtained similarly. The third one follows the second one by applying $ii)$ of Proposition~\ref{krew1bis}, using the fact that $w$ has full support.
\end{proof}

\begin{lemm}  \label{lemma_scalprodgeq0ter}
 Let $i,j$ be such that either $1\leq i<j \leq k$ or $k+1\leq i<j \leq n$. Then we have:
 \[
  \langle r(t_i) | r(t_j) \rangle \geq 0.
 \]
\end{lemm}

\begin{proof}
 We can focus on the case $1\leq i < j \leq k$, the other one is obtained similarly but with ${\mathfrak s}_{k+1},\dots,{\mathfrak s}_{n}$ instead of ${\mathfrak s}_1,\dots,{\mathfrak s}_k$.
 Using Proposition~\ref{inv_rootimages}, we have 
 \begin{align*}
    r(t_i) & = r({\mathfrak s}_k \cdots {\mathfrak s}_{k+1-i} \cdots {\mathfrak s}_k) = {\mathfrak s}_k
 \cdots {\mathfrak s}_{k+2-i} \big( r( {\mathfrak s}_{k+1-i} ) \big)  \\
                & =  - {\mathfrak s}_k \cdots {\mathfrak s}_{k+1-i} \big( r( {\mathfrak s}_{k+1-i} ) \big).
 \end{align*}
 The second equality follows from the fact that  ${\mathfrak s}_{k+1-i} \notin \Inv_R ( {\mathfrak s}_k \cdots {\mathfrak s}_{k+2-i} )$ since they have disjoint support in $\Gamma(w)$.
 Similarly, 
 \begin{align*}
    r(t_j) &= r({\mathfrak s}_k \cdots {\mathfrak s}_{k+1-j} \cdots {\mathfrak s}_k)  \\
                &= {\mathfrak s}_k \cdots {\mathfrak s}_{k+1-i} \big( r( {\mathfrak s}_{k-i} \cdots {\mathfrak s}_{k-j} \cdots 
{\mathfrak s}_{k-i} ) \big).
 \end{align*}
 Since $ {\mathfrak s}_k \cdots {\mathfrak s}_{k+1-i} $ preserves the scalar product, it remains to show that
 \[
   \langle r( {\mathfrak s}_{k+1-i} ) | r( {\mathfrak s}_{k-i} \cdots {\mathfrak s}_{k+1-j} \cdots {\mathfrak s}_{k-i} ) \rangle \leq 0.
 \] 
This follows from the fact that ${\mathfrak s}_1,\dots,{\mathfrak s}_k$ is a simple system, since the positive root $r( {\mathfrak s}_{k-i} \cdots {\mathfrak s}_{k+1-j} \cdots {\mathfrak s}_{k-i} )$ is a positive linear combination of the roots 
 $r({\mathfrak s}_{k-i})$, $\dots$, $r({\mathfrak s}_{k+1-j})$.
\end{proof}

%
%
%
%
%

\begin{lemm} \label{lemma_scalprodgeq0}
 Let $v\in NC(W,c)$ with $\rk(v)=n-2$, and let $x_1,x_2 \in T$ such that $c=v x_1 x_2$.
 Suppose $v\sqsubsetdot vx_1 \lldot c$. 
 Then $\langle r(x_1) | r(x_2) \rangle \geq 0$. 
\end{lemm}

\begin{proof}

The cardinality of $\supp(v)$ is at least $n-2$ since $\rk(v)=n-2$. We treat separately its possible values.

\noindent
{\bf Case 1: } $\# \supp(v)=n$.

Then $v\in NCF(W,c)$, so that we can apply Proposition~\ref{krew1bis} to the relation $v\sqsubsetdot vx_1$. We get $x_2\lldot x_1x_2$, therefore $x_1,x_2$ is not a simple system in $\Gamma(x_1x_2)$, and $\langle r(x_1) | r(x_2) \rangle > 0$.

\noindent
{\bf Case 2: } $\# \supp(v)=n-1$.

 Since $v$ is not full, there exists an interval partition $v'$ such that $v \lessdot v' \sqsubsetdot c$. With $v' = vx_3$ and $c=v'x_4$, 
 we have $vx_3x_4=c$, $v \lessdot vx_3 \sqsubsetdot c$, and $x_3x_4=x_1x_2$.
 Since $\# \supp(v)=n-1>\rank(v)$, we see that $v$ is not an interval partition, so that $v \lldot vx_3$. 
 By Proposition~\ref{krew1bis}, we get $K(vx_3)\sqsubsetdot K(v)$, i.e.,~$x_4\sqsubsetdot x_3x_4$.
 Now by Lemma~\ref{simplification}, $v \lldot vx_3 \sqsubsetdot vx_3x_4$ implies $x_3 \leq vx_3\sqsubsetdot vx_3x_4$ and $x_3\sqsubsetdot x_3x_4$ therefore 
  $x_3$ and $x_4$ are the simple generators of $\Gamma(x_3x_4)$.  
 
 If $\{x_1,x_2\} \neq \{x_3,x_4\}$, we get $\langle r(x_1) | r(x_2) \rangle > 0$ by uniqueness of the simple system
 $\{x_3,x_4\}$. Otherwise, we have $x_1=x_4$, indeed $x_1\neq x_3$ since $v\sqsubsetdot vx_1$ and $v \lldot vx_3$.
 It follows that $x_2=x_3$, and $x_3x_4=x_1x_2 $ becomes $x_2x_1=x_1x_2 $ so that
  $\langle r(x_1) | r(x_2) \rangle = 0$. 

\noindent
{\bf Case 3: } $\# \supp(v)=n-2$.
 
To avoid multiple indices, assume that the simple  reflections of $W$ are indexed so that $c=s_1 \cdots s_n$. 
Let $J=\supp(v)=S\setminus\{i,j\}$ (with $i<j$) then $v$ has length  $\ell_T(v)=n-2$ and $v\in W_J$  therefore $v$ is the Coxeter element of $W_J=\Gamma(v)$. It follows that   $v$ is an interval partition and $v=s_1 \cdots\hat s_i \cdots \hat s_j \cdots s_n$ (where $s_i$ and $s_j$ are omitted).
Let
\[
   x_3 = s_n \cdots s_{j+1} s_{j-1} \cdots s_i \cdots s_{j-1} s_{j+1} \cdots s_n, \qquad x_4 = s_n \cdots s_j \cdots s_n.
\]
then $vx_3x_4=c$ and  $x_3x_4=x_1x_2$. By an argument similar to that in the previous lemma, we have
\[
  \langle r(x_3) | r(x_4) \rangle = \langle r(s_{j-1} \cdots s_i \cdots s_{j-1}) | r(s_j) \rangle \leq 0
\]
therefore $x_3$ and $x_4$ are the simple generators of $\Gamma(x_1x_2)$. The end of the proof is as in the previous case
(here $x_1 \neq x_3$ because $vx_1 \lldot c$ and $vx_3 \sqsubsetdot c$).
\end{proof}

\begin{lemm}  \label{lemma_scalprodgeq0bis}
 We have $\langle r(t_i) | r(t_j) \rangle \geq 0$ if $1\leq i \leq k$ and $k+1\leq j \leq n$.
\end{lemm}

\begin{proof}
 By Lemma~\ref{lemma_relationsdot}, we have $wt_i\sqsubsetdot w \lldot wt_j$.
 Then we can apply Lemma~\ref{lemma_scalprodgeq0} in the subgroup $\Gamma(wt_j)$ to get the result.
\end{proof}

\begin{prop}  \label{proppsi}
 We have $\Psi(w)\in\Upsilon^+(W,c)$.
\end{prop}

\begin{proof}
 We use the criterion in Proposition~\ref{clust_pos}. Since $t_1\cdots t_n=c$, we have $t_it_j\leq c$ if $i<j$. 
 The conditions on the scalar product are given by Lemmas~\ref{lemma_scalprodgeq0ter} and \ref{lemma_scalprodgeq0bis}. So $r(t_i) \mathrel{ \|_c } r(t_j)$ holds for $1\leq i < j \leq n$ and the result follows.
\end{proof}

Let us now describe the inverse map.
Let $F=\{t_1,\ldots, t_n\}$ be a face of the positive cluster complex.
By Proposition~\ref{indexingface}, we can assume that the $t_i$ are ordered so that $t_1\ldots t_n=c$. Moreover, by Proposition~\ref{commCoxeter}, all orderings of the $t_i$ such that this property holds true are obtained from this ordering by applying commutation relations among the $t_i$. As before we let $u_i=t_1\ldots t_i$.

\begin{lemm}\label{commorder1}
If $u_{i-1}\ll u_{i}\sqsubset u_{i+1}$ then $t_it_{i+1}=t_{i+1}t_i$.
\end{lemm}
\begin{proof}
By Proposition~\ref{simplification} one has $t_i\sqsubset t_it_{i+1}$. Applying the Kreweras complement in $\Gamma(u_{i+1})$ and $i)$ of Proposition~\ref{krew1bis} we obtain $t_{i+1}\sqsubset t_it_{i+1}$. It follows that $t_i,t_{i+1}$ form a simple system in $\Gamma(t_it_{i+1})$ and $\langle r(t_i)|r(t_{i+1})\rangle\leq 0$.  Since $t_i\mathrel{ \|_c}t_{i+1}$ we have $\langle r(t_i)|r(t_{i+1})\rangle\geq 0$ , it follows that $\langle r(t_i)|r(t_{i+1})\rangle=0$ and $t_i,t_{i+1}$ commute.
\end{proof}

\begin{lemm}\label{commorder2}
Let $t,t'\in T$ be such that $tt'=t't$ and $w\preceq^{(1)}wt\preceq^{(2)}wt't$ then $w\preceq^{(2)}wt'\preceq^{(1)}wt't$, where $\preceq^{(1)},\preceq^{(2)}$ denotes any combination of the orders $\sqsubset,\ll$.
\end{lemm}

\begin{proof}
This follows from Corollary~\ref{refcommute}.
\end{proof}

Using Lemmas~\ref{commorder1} and \ref{commorder2} we can use commutation relations between the $t_i$ to move all the $\ll$ to the right and assume that, for some $k$, one has
\[ 
  e\sqsubset u_1\sqsubset u_2\sqsubset \ldots\sqsubset u_{k}\ll u_{k+1}\ll \ldots\ll u_n=c.
\] 
Moreover, by Proposition~\ref{commCoxeter}, $k$ and $u_k$ are uniquely determined by this requirement.  It is then easy to check
\[
\{t_1,\ldots, t_k\}=\Inv_R(u_k)\cap\Gamma(u_k), \text{ and } \{t_{k+1},\ldots, t_n\}=\Inv_L(K(u_k))\cap\Gamma(K(u_k)),
\]
therefore $\Psi_{W,c}(u_k)=F$.

\subsection{The bijection between faces and intervals}

Consider a positive face $ F = \{ t_1,\dots,t_k \} \in \Upsilon^+(W,c) $. By Proposition~\ref{indexingface}
we can assume that the elements are  indexed so that $t_1\cdots t_k \leq c$. Let also $w = t_1\cdots t_k $.

\begin{prop} \label{face_to_cluster}
In the situation described above, $F$ is a cluster in $\Upsilon^+( \Gamma(w) , w )$.
\end{prop}

\begin{proof}
First note that $\Upsilon^+( \Gamma(w) , w )$ is well defined by Proposition~\ref{reading1prodsimples}.
We have $t_it_j\leq w$ if $i<j$ since $w = t_1\cdots t_k $. Moreover, we have 
$\langle r(t_i) | r(t_j) \rangle \geq 0 $ since $F \in \Upsilon^+(W,c) $.  
So $F \in \Upsilon^+( \Gamma(w) , w )$ by Proposition~\ref{clust_pos}, and it is a cluster since $\# F = \ell_T(w)$ is the rank of $\Gamma(w)$.
\end{proof}

\begin{theo}
With the notation as above, the map
\[
   F \mapsto  \big( \Psi_{\Gamma(w),w}^{-1} ( F ) , w \big)
\]
is a bijection from $\Upsilon^+(W,c)$ to the set of pairs $v,w \in  NC(W,c)$ such that $v\ll w$.
\end{theo}

\begin{proof}
First note that $\Psi_{\Gamma(w),w}^{-1} ( F ) $ is well defined by Proposition~\ref{face_to_cluster}.
By properties of the bijection $\Psi$, we have $\Psi_{\Gamma(w),w}^{-1} ( F ) \in  NCF( \Gamma(w) , w )$.
By Proposition~\ref{llfull}, this means $\Psi_{\Gamma(w),w}^{-1} ( F ) \ll w $.

We can describe the inverse bijection. To a pair $v,w \in  NC(W,c)$
such that $v\ll w$, we associate $\Psi_{\Gamma(w),w}(v)$.
Once we know that $\Psi$ is a bijection, it is clear that we have two inverse bijections.
\end{proof}

The construction can be made more explicit. Let $u_i = t_1 \cdots t_i$ for $0\leq i \leq k$. Up to some commutation among the $t_i$, we can assume 
\[
  u_0 \sqsubsetdot \dots \sqsubsetdot u_j \lldot u_{j+1} \lldot \dots \lldot u_k.
\]
 Then the image of $F$ is $(u_j,u_k)$.
In the other direction, let $v,w\in NC(W,c)$ with $v\ll w$, $\rk(v)=j$ and
$\rk(w)=k$. We write $v$ and $v^{-1} w$
as a product of their associated simple reflections:
\[
  v = {\mathfrak s} _1 \cdots {\mathfrak s} _j, \qquad v^{-1}w = {\mathfrak s} _{j+1} \cdots {\mathfrak s}_{k}.
\]
Then the inverse image of $(v,w)$ is $\{t_1,\dots,t_k\}$ where
we define $t_1,\dots,t_k$ as in the definition of $\Psi$:
\[
  t_i = \begin{cases}
             {\mathfrak  s}_j \cdots {\mathfrak  s}_{j+1-i} \cdots {\mathfrak  s}_j & \text{ if } 1\leq i \leq j, \\
            {\mathfrak  s}_{j+1} \cdots {\mathfrak  s}_{k+1+j-i} \cdots {\mathfrak  s}_{j+1} & \text{ if } j+1\leq i \leq k.
        \end{cases}
\]

Also, an immediate consequence of the construction is the following.

\begin{defi} \label{defisqr}
Let $F\in \Upsilon^+(W,c)$, and write $F=\{t_1,\dots,t_k\}$ such that $t_1\cdots t_k \leq c$.
Then we define (number of ``square'' relations): 
\[
  \sqr(F) = \#\big\{ i \; : 0\leq i <k \text{ and } t_1\cdots t_i \sqsubsetdot t_1\cdots t_{i+1} \big\}.
\]
This map is extended to $F\in \Upsilon(W,c)$ by requiring $\sqr(F)=\sqr(F\cap \Pi)$.
\end{defi}

By Lemma~\ref{commorder1} the number $\sqr(F)$ does not depend on the way we order $F$, as long as $t_1\cdots t_k \leq c$.

\begin{prop} \label{bij_preservedstats}
We have:
\[
   \sum_{F\in\Upsilon^+(W,c)} y^{\sqr(F)} z^{\# F }  = \sum_{\substack{\alpha,\beta\in NC(W,c) \\ \alpha\ll \beta}} y^{\rk(\alpha)} z^{\rk(\beta)}.
\]
\end{prop}

However, this identity is not related to the $F=M$ theorem in a straightforward way, as the left hand side is seemingly unrelated to the polynomial $F(x,y)$.  This will be clarified in Section~\ref{sec_genFHM}.

\subsection{\texorpdfstring{Bijection between faces of the cluster complex and intervals for $\sqsubset$ in the case of a bipartite Coxeter elements}{Bijection between faces of the cluster complex and intervals for [ in the case of a bipartite Coxeter elements}}
\label{bijsqsub}
When $c=c_+c_-$ is a bipartite Coxeter element, one can give a bijection between the intervals for $\sqsubset$ and faces of the cluster complex. 
For this we need   a preliminary result. 

Let $u\ll v$ then,  by applying Proposition \ref{bipcomp1}, we get $Lv\sqsubset Lu$. It follows that 
$S(\Gamma(Lv))\subset S(\Gamma(Lu))$. Let $w\in \Gamma(Lu)$ be the element corresponding to the set $S(\Gamma(Lu))\setminus S(\Gamma(Lv))$ then one has $w\sqsubset Lu$. Let us  denote by
$\psi(u,v)$ the pair $(x,y)=(w,Lu)$. It is clear that the pair $(u,v)$ can be retrieved from $(x,y)$ therefore the map $\psi$ is injective.
\begin{prop}\label{bijllsqsub}
The map $\psi$ yields a bijection between intervals $u\ll v$ and intervals $x\sqsubset y$ such that $\underline x=e$.

\end{prop}

\begin{proof}
 Let $u\ll v$ and $(w,Lu)=\psi(u,v)$, we prove  that $\underline w=e$. Assume on the opposite that $\underline w\ne e$ then there exists $s\in S$ such that $s\in  S(\Gamma(w))\subset S(\Gamma(Lu))$ and $s\notin  S(\Gamma(Lv))=S(\Gamma(Lu))\setminus S(\Gamma(w))$. Since $s\leq Lu$ one has $sLu\leq Lu$ and $(Lu)s\leq Lu$. It follows easily that $u\in W_{\langle s\rangle }$ and 
$s\in S(\Gamma(v))$ which contradicts $u\ll v$. This proves that $\psi$ maps intervals $u\ll v$ to intervals $x\sqsubset y$ with $\underline x=e$.

Conversely, let $x\sqsubset y$ be such that $\underline x=e$ and let $z$ be the element of $\Gamma(y)$ corresponding to $S(\Gamma(y))\setminus S(\Gamma(x))$ then 
$z\sqsubset y$. Let us prove that $Ly\ll Lz$, it will follow from the construction that $(x,y)=\psi(Ly,Lz)$ and therefore that $\psi$ is surjective. Let $J$ be the support of $Ly$ then one has $Ly\in W_J$ and one can write
$y=c_+^{J^c}(c_+^JLyc_-^J)c_-^{J^c}$ with obvious notations: e.g. $c_+^J$ is the product of simple reflections in $J\cap S_+$, etc. Since $(c_+^JLyc_-^J)\in W_J$ is follows that for $s\in J^c$ one has  $s\in S(\Gamma(y))$ moreover since  $\underline x=e$ one has $s\notin S(\Gamma(x))$ therefore  $s\in S(\Gamma(z))$. It follows that
$z=c_+^{J^c}\omega c_-^{J^c}$ with $\omega\in W_{J}$ moreover $\omega\sqsubset c_+^JLyc_-^J$ since $z\sqsubset y$. Since $Ly$ has full support in $W_J$ we can apply $ii)$ of Proposition~\ref{bipcomp1} in the group $W_J$ and conclude that $Lz=c_+^J\omega c_-^J\ll Ly$ as claimed.
\end{proof}

Let $\gamma$ be an interval partition, corresponding to the parabolic subgroup $W_J$ then one can write $\gamma=\gamma_+\gamma_-$ according to the bipartite decomposition. Let $v\in NC(W,c)$ be such that $\underline v=\gamma$, then 
$v=\gamma_+v'\gamma_-$ where $v'\in NC(W_{J^c},L\gamma)$ moreover this gives a bijection between $NC(W_{J^c},L\gamma)$ and the set of $v$ such that $\underline v=\gamma$.

Using the bijection $\psi^{-1}$ and composing with the bijection between intervals for $\ll$ and faces of the positive cluster complex we get 
 a bijection between intervals $v\sqsubset w$ for $\sqsubset$ with $\underline v=e$ and faces of the positive cluster complex $\Upsilon^+(W,c)$. It is an easy exercize to check that   an interval of height $k$ corresponds to a face of size $n-k$. This bijection is then easily extended to a bijection between intervals of $\sqsubset$ and faces of $\Upsilon(W,c)$
if $\underline v=\gamma$ add the set $J$ to the face in $\Upsilon^+(W_{J^c},c_{J^c})$.

\section{\texorpdfstring{Generalized $F=M$ and $H=M$ theorems}{Generalized F=M and H=M theorems}}
\label{sec_genFHM}

We show in this section that the relations between $F$-, $H$-, and $M$-polynomials can be proved and even generalized using the $I$-polynomial counting intervals for the order $\ll$.  Here we consider homogeneous polynomials on variables $(x_s)_{ s \in S}$ indexed by the simple reflections of $W$, and $y$, $z$.  Note that the existence of multivariate analog of the identities was suggested by Armstrong \cite[Open problem~5.3.5]{armstrong}.

In general we denote $\underline{x}$ the set of $x$ variables, leaving the index set implicit.  Moreover $\underline{x}+A$ denote that all the $x$ variables are shifted by $A$, and $\underline{x}$ is replaced by an expression to mean that all $x$ variables are specialized to this expression, etc.

Let 
\begin{align*}
  \I(\underline{x}, y, z) &= \sum_{\substack{\alpha,\beta\in NC(W,c) \\ \alpha\ll \beta}}  \Bigg( \prod_{ s\in S\backslash \supp(\beta)} x_s \Bigg) y^{ \rk(\alpha) } z^{\rk(\beta)-\rk(\alpha)},\\
  \F(\underline{x}, y, z) &= \sum_{F \in \Upsilon(W,c) } \Bigg( \prod_{ \delta \in F\cap (-\Delta) } x_{r(\delta)} \Bigg) \Big(\frac yz \Big)^{\sqr(F)} z^{\# F\cap \Pi} ,
\end{align*}
and also:
\begin{align*}
  \M(\underline{x}, y, z) &= \sum_{\substack{\alpha,\beta\in NC(W,c) \\ \alpha\leq \beta}} \mu(\alpha,\beta) \Bigg( \prod_{ s\in S\backslash \supp(\beta)} x_s \Bigg) y^{\rk(\alpha)} (-z)^{\rk(\beta)-\rk(\alpha)},\\
  \H(\underline{x}, y, z) &= \sum_{A \in  NN(W) } \Bigg( \prod_{ \delta \in A\cap \Delta } \Big(\frac {x_{r(\delta)}}z\Big) \Bigg) \Big(\frac yz\Big)^{\# A\cap( \Pi\backslash \Delta ) } z^{\#\supp(A)}.
\end{align*}
It is straightforward to check that they are all polynomials of total degree $n$. For example, note that we have $\sqr(f)\leq \# (F\cap \Pi)$ for $F\in\Upsilon(W,c)$ by definition of $\sqr$, so the power of $z$ in $\F(\underline{x}, y, z)$ is nonnegative, and we have $\rk(\alpha) \leq \#\supp(\alpha)$ for $\alpha\in NC(W,c)$ so the polynomials $\I(\underline{x}, y, z)$ and $\M(\underline{x}, y, z)$ have degree at most $n$.

The 2-variable polynomials from Section~\ref{triangles} are obtained as special cases, though by different specializations:
\begin{equation} \label{FHMspec}
\begin{aligned}
  F(x,y) &= \F(x,y,y), \quad H(x,y) = \H(x,y,1), \\
  M(y,z) &= \M(1,y,z), \quad I(y,z) = \I(1,y,z).
\end{aligned}
\end{equation}

\begin{theo} \label{theoFI}
  $\F(\underline{x},y,z) = \I(\underline{x}+1,y,z)$.
\end{theo}

\begin{proof}
First, note that
\[
  \F(\underline{x},y,z) = \sum_{J\subset S } \Bigg( \prod_{s\in J} x_s \Bigg) \F_{W_{S\backslash J}}(0,y,z).
\]
Also, expanding the products in $ \I(\underline{x}+1, y , z)$ gives:
\begin{align*}
  \I(\underline{x}+1, y , z) &= \sum_{\substack{ \alpha,\beta\in NC(W,c) \\ \alpha\ll \beta }}  \; \sum_{ J \subset S\backslash \supp(\beta) } \Bigg( \prod_{ s\in J } x_s \Bigg) y^{ \rk(\alpha) } z^{\rk(\beta)-\rk(\alpha)} \\
                            &= \sum_{ J \subset S } \; \sum_{\substack{ \alpha,\beta\in NC(W,c), \; \alpha\ll \beta \\ \supp(\beta)\subset S\backslash J }}  \Bigg( \prod_{s\in J} x_s \Bigg) y^{ \rk(\alpha) } z^{\rk(\beta)-\rk(\alpha)} \\
                            &= \sum_{ J \subset S }  \Bigg( \prod_{s\in J} x_s \Bigg) \I_{W_{S\backslash J}}(1,y,z).
\end{align*}
So it suffices to prove $\F(0,y,z)=\I(1,y,z)$, which is the content of Proposition~\ref{bij_preservedstats}.
\end{proof}

\begin{theo}
  $\M(\underline{x},y,z) = \I(\underline{x},y+z,z)$. 
\end{theo}

\begin{proof}
This follows from Proposition~\ref{refinementIM}. In $I_\beta(\frac yz+1) = M_\beta(\frac yz)$, 
multiply both sides by $\big(\prod_{s\in S\backslash \supp(\beta) } x_s \big) z^{\rk(\beta)}$ and sum over $\beta\in NC(W,c)$.
\end{proof}

\begin{theo}
  $\H(\underline{x},y,z) = \I(\underline{x}-y+1,y,z-1)$.
\end{theo}

\begin{proof}
First write the expansion:
\[
  \prod_{ \delta \in A\cap \Delta } \Big(\frac {x_{r(\delta)}+y}z\Big) = \sum_{J\subset A\cap \Delta} \Big(\frac yz \Big)^{\# (A\cap\Delta)\backslash J}  \prod_{\delta \in J}  \Big(\frac {x_{r(\delta)}}z\Big).
\]
It follows:
\begin{align*}
  \H( \underline{x}+y, y , z) &=  \sum_{A\in  NN(W)} \; \sum_{J\subset A\cap \Delta } 
  \Bigg( \prod_{\delta \in J}  \Big(\frac {x_{r(\delta)}}z\Big) \Bigg)
  \Big(\frac yz\Big)^{\# (A \backslash J )} z^{\#\supp(A)} \\
  &= \sum_{J \subset \Delta } \Bigg( \prod_{\delta \in J} x_{r(\delta)} \Bigg) \sum_{\substack{ A\in NN(W) \\ J\subset A }}
     \Big(\frac yz\Big)^{\# (A \backslash J )} z^{\#\supp(A) - \#J } \\
  &= \sum_{J \subset \Delta } \Bigg( \prod_{\delta \in J} x_{r(\delta)} \Bigg) \H_{ W_{S\backslash r^{-1}(J) } }(y,y,z).
\end{align*}
In the proof of the Theorem~\ref{theoFI}, we have seen that $\I(\underline{x}+1,y,z)$ admits a similar expansion on the $x$-variables.
Consequently, $\I(\underline{x}+1,y,z-1)$ also admits a similar expansion, and it remains to show that $\H(y,y,z)=\I(1,y,z-1)$.

On one side we have:
\[
  \H(y,y,z) = \sum_{A\in NN(W)} y^{\#A}z^{\#\supp(A)-\#A} = \sum_{I\subset S} \; \sum_{k=1}^{\# I}  \Nar^+_k(W_I) y^k z^{\# I - k}.
\]
On the other side,
\begin{align*}
  \I(1,y,z-1) &= \sum_{\substack{ \alpha,\beta\in NC(W,c) \\ \alpha\ll\beta }} y^{\rk(\alpha)} (z-1)^{\rk(\beta)-\rk(\alpha)}  \\
              &= \sum_{\alpha\in NC(W,c)} y^{\rk(\alpha)} z^{\#\supp(\alpha) - \rk(\alpha) },
\end{align*}
where the last equality follows from the binomial theorem, since the $\beta$ such that $\beta\gg \alpha$ form a boolean lattice
whose maximal element is $\overline{\alpha}$, by Proposition~\ref{booleanideals0} and Corollary~\ref{booleanideals}.
This sum can again be expressed in terms of the numbers $\Nar^+_k(W_I)$, so that $\H(y,y,z)=\I(1,y,z-1)$.
\end{proof}

By combining the previous theorems, we get relations between $\F$-, $\H$-, and $\M$-polynomials:
\begin{align}
\label{FMgen}   \F(\underline{x},y,z) &=\M(\underline{x}+1,y-z,z), \\
\label{FHgen}    \F(\underline{x},y,z) &=\H(\underline{x}+y,y,z+1), \\
\label{HMgen}     \H(\underline{x},y,z) &=\M(\underline{x}-y+1,y-z+1,z-1). 
\end{align}

The next property is best seen on the $\H$-polynomial.

\begin{prop}
 $\H(\underline{x}-1,y,z)$ is homogeneous of degree $n$.
\end{prop}

\begin{proof}
 An element $A\in NN(W)$ can be uniquely written $A'\cup A''$ where $A'\cap \Delta = \varnothing$ and $A''\subset \Delta \backslash \supp(A')$.
 Note that $\#\supp(A) = \#\supp(A') + \# A''$. For a fixed $A'$, any subset $A''\subset \Delta \backslash \supp(A')$ is valid, so that
 the sum over $A''$ factorizes and gives:
\begin{align*}
   \sum_{\substack{ A\in  NN(W) \\ A\cap(\Pi\backslash\Delta)=A' }} \Bigg( \prod_{ \delta \in A\cap \Delta } \Big(\frac {x_{r(\delta)}}z\Big) \Bigg) \Big(\frac yz\Big)^{\# A\cap( \Pi\backslash \Delta ) } z^{\#\supp(A)} \\
    = \Bigg( \prod_{\delta\in \Delta\backslash \supp(A')} (1+x_{r(\delta)}) \Bigg) \Big(\frac yz\Big)^{\# A' } z^{\#\supp(A')}.
\end{align*}
After a change of variables $\underline{x}\to \underline{x}-1$, the latter expression is homogeneous.
Since $\H(\underline{x}-1,y,z)$ is obtained by summing over $A'$, it is too.
\end{proof}

This homogeneity implies that the 3-variable polynomial $\H(x,y,z)$ can be expressed in terms of the 2-variable polynomial $H(y,z) = \H(1,y,z)$.
It also implies that the $\F$-, $\I$-, and $\M$-polynomials also become homogeneous after suitable shifts in the variables.
So the other 3-variable polynomials can also be expressed in terms of their 2-variable specialization from Section~\ref{triangles}.

To recover the $F=M$ identity in Section~\ref{triangles} from the present results, first use the homogeneity of $\F(x-1,y,z-1)$ 
(obtained from the previous proposition together with \eqref{FHgen}) to write:
\[
  \F(x-1,y,z-1) = x^n \F(0,\tfrac yx , \tfrac zx - 1 ),
\]
then the substitution $(x,z)\to (x+1,y+1)$ gives:
\[
  \F(x,y,y) = (1+x)^n \F(0,\tfrac{y}{1+x},\tfrac{y-x}{1+x}) = (1+x)^n \M(1,\tfrac{x}{1+x},\tfrac{y-x}{1+x})
\]
where the last equality comes from \eqref{FMgen}. Then using \eqref{FHMspec}, we recover the relation between $F(x,y)$ and $M(x,y)$ in \eqref{eqFHM}.
Similarly, the homogeneity of $\H(\underline{x}-1,y,z)$ gives:
\begin{equation}  \label{Hhomogeneity}
  \H(x-1,y,z) = z^n \H(\tfrac xz -1 , \tfrac yz , 1).
\end{equation}
Then the substitution $(x,z)\to(x+y+1,y+1)$ gives:
\[
  \H( x+y, y, y+1 ) = (1+y)^n \H( \tfrac{x}{1+y}, \tfrac{y}{1+y} , 1 ).
\]
Using \eqref{FHgen} on the left hand side, then \eqref{FHMspec}, we recover the relation between $F(x,y)$ and $H(x,y)$ in \eqref{eqFHM}.

\begin{prop}
 $\H(\underline{x}-1,y,z)$ is symmetric in $y$ and $z$.
\end{prop}

\begin{proof}
Similar to the expansions we have seen in this section, we have:
\[
  \H(\underline{x},y,z) = \sum_{J\subset S} \Bigg( \prod_{s\in J} x_s \Bigg)   \H_{W_{S\backslash J}}(0,y,z).
\]
So we can just prove that $\H(0,y,z)$ is symmetric. By \eqref{Hhomogeneity}, we have
\[
  \H(0,y,z) = z^n \H(\tfrac 1z -1, \tfrac yz , 1) = z^n H(\tfrac 1z -1, \tfrac yz).
\]
By \eqref{symm_FH}, the latter is equal to $y^n H(\tfrac 1y -1 , \tfrac zy )$, which is precisely
the same up to exchanging $y$ and $z$.
\end{proof}

It is natural to look for an involution on nonnesting partitions that would prove the symmetry of $\H(\underline{x},y,z)$.
This is the subject of conjectures by Panyushev \cite[Conjecture~6.1]{panyushev}, see also \cite[Section~5]{chapoton2}.
To our knowledge this is still an open problem.

\section*{Acknowledgement}

We thank Cesar Ceballos for explanations about the subword complex.

\setlength{\parindent}{0mm}

\end{document}